\documentclass[11pt,reqno]{amsart}

\usepackage{amssymb}
\usepackage{etex} 
\usepackage{tikz-cd}
\usepackage{mathtools}
\usepackage{xifthen}
\usepackage{stmaryrd}  
\usepackage{enumitem}

\usepackage[left=2.6cm,right=2.6cm,top=2.7cm,bottom=2.8cm]{geometry}
\usepackage{microtype}

\newenvironment{claim}[1][]{%
\begin{description}[leftmargin=0pt,before={}]%
\ifthenelse{\isempty{#1}}{\item[Claim]}{\item[Claim {#1}]}
}%
{\end{description}}

\definecolor{linkblue}{RGB}{1,1,190}
\definecolor{citegreen}{RGB}{1,190,1}
\usepackage[linkcolor=linkblue
           ,citecolor=citegreen
           ,ocgcolorlinks
           ,bookmarksopen=True,  
           ]{hyperref}
\usepackage[ocgcolorlinks]{ocgx2} 
\makeatletter
  \AtBeginDocument{
    \hypersetup{
      pdftitle  = {\@title},
      pdfauthor = {\authors}     
    }
  }
\makeatother

\newcommand{\raw}{\rightarrow}

\newcommand{\ra}{\rightarrow}

\DeclarePairedDelimiter{\card}{\lvert}{\rvert}

\newcommand{\x}{\mathbf{x}}

\newcommand{\y}{\mathbf{y}}
\newcommand{\Mut}{\mathcal{M}}

\newcommand{\Z}{\mathbb{Z}}
\newcommand{\Q}{\mathbb{Q}}

\newcommand{\Spec}{\mathrm{Spec}}

\newcommand{\ps}{V}
\newcommand{\prim}{\mathfrak{p}}
\newcommand{\qrim}{\mathfrak{q}}

\newcommand{\Xx}{\mathfrak{X}}
\newcommand{\cI}{\mathcal{I}}
\newcommand{\cF}{\mathcal{F}}
\newcommand{\val}{\mathsf{v}}
\newcommand{\cH}{\mathcal{H}}
\newcommand{\cC}{\mathcal{C}}
\newcommand{\cU}{\mathcal{U}}
\newcommand{\fa}{\mathfrak{a}}
\newcommand{\fb}{\mathfrak{b}}

\newcommand{\bgcd}{\mathsf d}  
\newcommand{\dcount}{\mathsf c}
\newcommand{\vexp}{\mathsf e}
\newcommand{\ptns}[1]{P({#1})}  
\newcommand{\PP}[2]{\prim_{#1}({#2})}
\newcommand{\fp}{\mathfrak p}
\newcommand{\fq}{\mathfrak q}
\DeclareMathOperator{\im}{im}
\DeclareMathOperator{\ord}{ord}

\newcommand{\vprod}{\sideset{}{\mkern -2mu\raisebox{0.2em}{$\vphantom{\prod}^v$}} \prod}

\newcommand{\inv}{\mathbf{inv}}

\newtheorem{theo}{Theorem}[section]
\newtheorem{lemm}[theo]{Lemma}
\newtheorem{coro}[theo]{Corollary}
\newtheorem{prop}[theo]{Proposition}

\newtheorem{maintheo}{Theorem}

\theoremstyle{definition}
\newtheorem{exam}[theo]{Example}
\newtheorem{rema}[theo]{Remark}
\newtheorem{ques}[theo]{Question}
\newtheorem{defi}[theo]{Definition}

\setlist[enumerate,1]{label=\textup{(\arabic*)}, ref=\textup{(}\arabic*\textup{)}, leftmargin=0.75cm}
\setlist[itemize,1]{leftmargin=0.75cm}
\newlist{equivenumerate}{enumerate}{1}
\setlist[equivenumerate,1]{%
  label=\textup{(\alph*)},
  ref=\textup{(}\alph*\textup{)},
  leftmargin=0.75cm
}

\begin{document}


\title{Factoriality and class groups of cluster algebras}

\author{Ana Garcia Elsener}
\address{University of Graz\\
        Institute for Mathematics and Scientific Computing\\
         NAWI Graz\\
         Heinrichstrasse 36\\
         8010 Graz, Austria}
\email{elsener@mdp.edu.ar}

\author{Philipp Lampe}
\address{Department of Mathematical Sciences\\
Durham University\\
Science Laboratories\\
South Road\\
Durham, DH1 3LE\\
United Kingdom}
\email{p.b.lampe@kent.ac.uk}

\author{Daniel Smertnig}
\address{Department of Mathematics\\
6188 Kemeny Hall\\
Dartmouth College\\
Hanover, NH 03755-3551\\
USA}
\email{daniel@smertnig.at}

\begin{abstract}
  Locally acyclic cluster algebras are Krull domains.
  Hence their factorization theory is determined by their (divisor) class group and the set of classes containing height-$1$ prime ideals.
  Motivated by this, we investigate class groups of cluster algebras.
  We show that any cluster algebra that is a Krull domain has a finitely generated free abelian class group, and that every class contains infinitely many height-$1$ prime ideals.
  For a cluster algebra associated to an acyclic seed, we give an explicit description of the class group in terms of the initial exchange matrix.
  As a corollary, we reprove and extend a classification of factoriality for cluster algebras of Dynkin type.
  In the acyclic case, we prove the sufficiency of necessary conditions for factoriality given by Geiss--Leclerc--Schr\"oer.
\end{abstract}

\maketitle

\section*{Introduction and summary of results}

Cluster algebras were introduced and studied in a series of articles by Fomin and Zelevinsky in \cite{FZ1,FZ2,FZ4} and by Berenstein--Fomin--Zelevinsky in \cite{BFZ}.
They admit connections to several branches of mathematics such as representation theory, geometry, and combinatorics.
A large number of articles in cluster theory study linear bases of cluster algebras. In most applications however---for example in representation theory and in hyperbolic geometry---the multiplicative structure is of prime importance.
For this reason we study ring-theoretic properties of cluster algebras, in particular those related to their factorization theory. The factorization theory of cluster variables also plays a crucial role in the proof of a main theorem about cluster algebras, Fomin--Zelevinsky's Laurent phenomenon, see \cite{FZ1}.
Further, Goodearl--Yakimov use special prime elements in their construction of quantum cluster algebra structures on quantum nilpotent algebras in \cite{GY1}, and in a commutative setting in their construction of cluster algebra structures on symmetric Poisson nilpotent algebras in \cite{GY2}.

Throughout the paper we consider cluster algebras of geometric type, allowing frozen variables.
However, we always assume that \emph{all frozen variables are invertible}.
The only exception to this is Section~\ref{sec:noninvertible}, where we discuss the case of non-invertible frozen variables.

A locally acyclic cluster algebra $A$ is a noetherian integrally closed domain, and therefore a Krull domain.
Krull domains are classical objects in commutative algebra (see \cite{Fos}).
They have been one of the main objects of study in factorization theory.
Thus, the factorization theory of a Krull domain $A$ is very well understood (see \cite{GHK,G}).
It is governed by the (divisor) class group $\cC(A)$ and the subset of classes of $\cC(A)$ containing height-$1$ prime ideals.
In this paper we investigate class groups of cluster algebras that are Krull domains.
As a domain is factorial if and only if it is a Krull domain with trivial class group, this constitutes a generalization of the study of factoriality of cluster algebras.

Factoriality of cluster algebras has been investigated before by Geiss--Leclerc--Schr\"oer in \cite{GLS}, who, in particular, have given necessary conditions for a cluster algebra to be factorial.
Factoriality in cluster algebras of Dynkin type $A$, $D$, $E$ was classified in the (unpublished) preprint \cite{L1} by Lampe.
The connection with Krull domains was first observed in \cite{L2}.

The results in the present paper give a full description of the class group of an acyclic cluster algebra in terms of its initial seed.
Class groups of Dynkin types (and other important families of acyclic seeds) are easily computed from these results; in particular, we recover the classification of factoriality for Dynkin type $A$, $D$, $E$ and extend it to all Dynkin types (see Corollary~\ref{coro:dynkin}).
Cluster algebras with an acyclic seed and principal coefficients turn out to always be factorial (see Corollary~\ref{coro:principal-factorial}).

We show that the necessary condition for factoriality of Geiss--Leclerc--Schr\"oer is sufficient for an acyclic seed (see Theorem~\ref{Thm:AcyclicFactorial}).
For, not necessarily acyclic, cluster algebras that are Krull domains, our first main result shows that their class group is always a finitely generated free abelian group.

\begin{maintheo}\label{Thm:MainGeneral}
  Let $\Sigma = (\x,\y,B)$ be a seed with exchangeable variables $\x=(x_1,\ldots, x_n)$ and frozen variables $\y=(x_{n+1}, \ldots, x_{n+m})$.
  Let $A=A(\Sigma)$ be the cluster algebra associated to $\Sigma$, in which we assume that all frozen variables are invertible.
  Suppose that $A$ is a Krull domain, and let $t \in \Z_{\ge 0}$ denote the number of height-$1$ prime ideals that contain one of the exchangeable variables $x_1$, $\ldots\,$,~$x_n$.
  Then the class group of $A$ is a free abelian group of rank $t - n$.

  If $n+m > 0$, that is $A \ne K$, then each class contains exactly $\card{K}$ height-$1$ prime ideals.
\end{maintheo}

Since every locally acyclic cluster algebra is a Krull domain, the previous theorem is applicable to them.
The theorem also has a generalization allowing non-invertible frozen variables, see Theorem~\ref{Thm:MainGeneralNonInvertible}.

These results imply a strong dichotomy between the factorization theory of a factorial and a non-factorial cluster algebra: in a non-factorial cluster algebra that is a Krull domain, all arithmetical invariants are infinite and any finite subset $L \subseteq \Z_{\ge 2}$ can be realized as a set of lengths of some element (see Corollary~\ref{coro:lengths}).
We also give an example of a cluster algebra that is not a Krull domain in Theorem~\ref{Thm:Markov}.

We now state a simplified version of our second main result on class groups of cluster algebras, which allows us to determine the rank of the class group in the acyclic case.
We restrict to the base ring $\Z$ and assume that $\Sigma$ does not have any isolated exchangeable indices; see Theorem~\ref{Thm:MainAcyclic} for the general result.
Before stating the theorem we need to introduce some notation.

Let $B=(b_{ij})$ denote a $(n+m)\times n$ integer matrix with skew-symmetrizable principal part.
To $B$ we associate an ice quiver $\Gamma(B)$ having set of vertices $[1,n+m]$, with $[n+1,n+m]$ frozen vertices.
When the principal part of $B$ is skew-symmetric, $\Gamma(B)$ coincides with the usual quiver associated to $B$ in cluster theory.
The exchange matrix $B$ is \emph{acyclic} if the full exchangeable sub-quiver of $\Gamma(B)$ does not contain an oriented cycle.
An index $i \in [1,n]$ is \emph{isolated} if $i$ has no neighbors in $\Gamma(B)$.
We write $\bgcd_j = \gcd( b_{ij} \mid i \in [1,n+m] )$ for the greatest common divisor of the $j$-th column of $B$.
Two indices $i$,~$j \in [1,n]$ are \emph{partners} if their exchange polynomials share a non-trivial common factor;
this can be expressed in terms of the exchange matrix $B$ (see Definition~\ref{defi:partners} and the discussion following it).
Partnership is an equivalence relation on $[1,n]$ and the equivalence classes are called \emph{partner sets}.

\begin{maintheo}[Special case of Theorem~\ref{Thm:MainAcyclic}] \label{Thm:MainAcyclicSimplified}
  Let $\Sigma = (\x,\y,B)$ be an acyclic seed with exchangeable variables $\x=(x_1,\ldots, x_n)$ and frozen variables $\y=(x_{n+1}, \ldots, x_{n+m})$.
  Let $A=A(\Sigma)$ be the cluster algebra associated to $\Sigma$, with base ring $\Z$ and all frozen variables invertible.
  Assume that $\Sigma$ does not contain any isolated exchangeable indices.
  For a partner set $\ps \subseteq [1,n]$ and $d \in \Z_{\ge 1}$, let $\dcount(\ps,d)$ denote the number of $i \in V$ for which $d$ divides $\bgcd_i$.

Then, the class group of $A$ is a finitely generated free abelian group of rank
  \[
    r= \sum_{V} \sum_{\substack{d \in \Z_{\ge 1}\\ d \text{ odd}}} \big(2^{\dcount(\ps,d)} - 1)  - \card{V},
  \]
  where the outer sum is taken over all partner sets $V \subseteq [1,n]$.
\end{maintheo}

The main tools we rely on are the Laurent phenomenon for cluster algebras and Nagata's Theorem, describing the behavior of class groups of Krull domains under localization.
In the case of an acyclic seed, we use a presentation of the cluster algebra due to Berenstein--Fomin--Zelevinsky, from \cite{BFZ}, to determine the height-$1$ prime ideals containing one of the initial cluster variables.

The paper is organized as follows.
In Section~\ref{sec:prelim} we recall basic definitions and results on cluster algebras arising from skew-symmetrizable matrices, on the multiplicative ideal theory of Krull domains, as well as some notions from factorization theory.
In Section~\ref{Section:exch-partners} we study the factorization properties of exchange polynomials and introduce the key notion of \emph{partner sets}.
Section~\ref{sec:class-groups} contains our main result on cluster algebras that are Krull domains (Theorem~\ref{Thm:MainGeneral}).
It also contains Subsection \ref{subsection F-poly}, where we show that cluster variables in factorial cluster algebras give rise to irreducible $F$-polynomials.

In Sections~\ref{sec:acyclic-prime} and \ref{sec:class-group-acyclic} we specialize to the case of acyclic seeds.
In Section~\ref{sec:acyclic-prime} we determine the height-$1$ prime ideals containing one of the initial cluster variables.
Using this, in Section~\ref{sec:class-group-acyclic}, we determine the class group of a cluster algebra associated to an acyclic seed.
The main theorems here are Theorem~\ref{Thm:AcyclicFactorial}, Theorem~\ref{Thm:MainAcyclic}, and Corollary~\ref{coro:dynkin}; we also give several easier to apply corollaries and work out examples.

In Section~\ref{sec:markov} we show that the cluster algebra associated to the Markov quiver is not a Krull domain.
Section~\ref{sec:noninvertible} contains partial generalizations of the main theorems to non-invertible frozen variables.
The final Section~\ref{sec:further} contains some further questions to investigate.

\section{Preliminaries}
\label{sec:prelim}
Throughout this article, let $K$ be a field of characteristic zero, or $K= \mathbb{Z}$.
From the end of Section~\ref{sect 1} on throughout the remainder of the paper,  if the base ring $K$ is a field, we assume that the underlying ice quiver $\Gamma(B)$ of the exchange matrix $B$ of our cluster algebra has no isolated exchangeable vertices.
This simplifies the statements of many results, without restricting their generality.
See Remark~\ref{Rem:FreezeTrivial} below.

A \emph{domain} is a nonzero unital commutative ring in which $0$ is the only zero-divisor.
For a domain $A$ we denote by $A^\times$ its group of units, by $A^\bullet=A\setminus \{0\}$ its monoid of nonzero elements, and by $\mathbf q(A)$ its quotient field.

\subsection{Cluster algebras}\label{sect 1}

In this subsection we recall basic notions about cluster algebras of geometric type, as introduced in the articles by Fomin--Zelevinsky \cite{FZ1,FZ2,FZ4} and by Berenstein--Fomin--Zelevinsky \cite{BFZ}. 

A good way to organize information about a cluster algebra is the notion of an ice quiver. Through all this work, the quivers considered will be \emph{$2$-acyclic}. This means that there are no oriented cycles of length one or two.

\begin{defi}[Ice quivers]
  \mbox{}

  \begin{enumerate}
    \item A \emph{quiver} $Q=(Q_0,Q_1)$ is a finite directed graph where $Q_0$ is the set of vertices and $Q_1$ the set of arrows. There are maps $s$,~$t \colon Q_1 \ra Q_0$ that indicate the source and the target of each arrow, respectively.
    \item An \emph{ice quiver} is a quiver $Q=(Q_0,Q_1)$ together with a partition of the vertex set $Q_0$ into \textit{exchangeable} and \textit{frozen} vertices.
      We also assume that there are no arrows between two frozen vertices.
  \end{enumerate}
\end{defi}

The following notions will become helpful.

\begin{defi}[Attributes of ice quivers]
  Let $Q=(Q_0,Q_1)$ be an ice quiver.
  \begin{enumerate}
    \item The \emph{exchangeable part} of $Q$ is the full subquiver on the set of exchangeable vertices.
    \item We say that $Q$ is \emph{acyclic} if the full subquiver on its exchangeable vertices does not contain any oriented cycle. 
    \item Two arrows $\alpha \neq \beta$ in $ Q_1$ are called \emph{parallel} if $s(\alpha)=s(\beta)$ and $t(\alpha)=t(\beta)$. 
    \item Let $i\in Q_0$. Elements in the set $N_{-}(i)=\{\,j\in Q_0 \ \vert \ \exists \alpha\colon j\to i \textrm{ in }Q_1\,\}$ are called \emph{predecessors}, elements in the set $N_{+}(i)=\{\,j\in Q_0\ \vert\  \exists \alpha\colon i\to j \textrm{ in }Q_1\,\}$ are called \emph{successors}, and elements in the set $N(i)=N_{-}(i)\cup N_{+}(i)$ are called \emph{neighbors}.
    \item A vertex $i\in Q_0$ is called a \emph{sink} if $N_{+}(i)=\emptyset$ and it is called a \emph{source} if $N_{-}(i)=\emptyset$. It is called \emph{isolated} if $N(i)=\emptyset$.
  \end{enumerate}
\end{defi}

Let $Q=(Q_0,Q_1)$ be an ice quiver with exchangeable vertices $[1,n]$ and frozen vertices $[n+1,n+m]$. With $Q$ we associate an integer $(n+m)\times n$ matrix $B=B(Q)=(b_{ij})$ with entries 
\begin{align*}
  b_{ij}=\lvert\{\,\alpha\in Q_1\ \vert \ s(\alpha)=i\textrm{ and }t(\alpha)=j\,\}\rvert-\lvert\{\,\alpha\in Q_1 \ \vert \ s(\alpha)=j\textrm{ and }t(\alpha)=i\,\}\rvert\in\mathbb{Z}.
\end{align*}
The $n\times n$ submatrix of an $(n+m)\times n$ matrix $B$ supported on rows $[1,n]$ is called the \emph{principal part} of $B$.
Note that the principal part of $B(Q)$ is skew-symmetric.

More generally, an $n\times n$ integer matrix $B=(b_{ij})$ is called \emph{skew-symmetrizable} if there exist positive integers $d_1$, $\ldots\,$,~$d_n$ such that $d_i b_{ij} = -d_j b_{ji}$ for all $i$,~$j \in [1,n]$.
We call an $(n+m)\times n$ integer matrix with $n$,~$m\in \mathbb{Z}_{\geq 0}$ and a skew-symmetrizable principal part an \emph{exchange matrix}.
Note that for every exchange matrix $B$ with skew-symmetric principal part there is exactly one ice quiver $Q$ with $B(Q)=B$ with exchangeable vertices $[1,n]$ and frozen vertices $[n+1,m]$.

\begin{rema}
  The notion of exchange matrices with a skew-symmetrizable principal part is more general than that of ice quivers, however it can be described using \emph{weighted directed graphs} (see \cite{FZ2}). We will not use this terminology.
 In this paper we will consider cluster algebras defined in terms of exchange matrices with a skew-symmetrizable principal part.
  However, many of our examples have skew-symmetric exchange matrices and can therefore be represented more nicely using ice quivers.
\end{rema}

Note that any skew-symmetrizable matrix is sign-skew-symmetric, that is, for any $i$, $j \in [1,n]$ either $b_{ij}=b_{ji}=0$ or $b_{ij}b_{ji} < 0$.

We associate to an exchange matrix $B$ the  ice quiver $\Gamma(B)$ in the following way. If $b_{ij} >0 $ we have $b_{ij}$ arrows from $i$ to $j$. If $i \in [n+1,m]$ is frozen and $b_{ij} <0$ we also add $-b_{ij}$ arrows from $j$ to $i$. Notice that in the case of matrices with skew-symmetric principal part, the exchange matrix $B$ and the ice quiver $\Gamma(B)$ carry the same information and one can be easily recovered from the other. That is, we can write $B(\Gamma(B))=B$.

We apply the notions of \emph{predecessors}, \emph{successors}, \emph{neighbors}, \emph{sources}, and \emph{sinks} to indices $i \in [1,n+m]$ by interpreting them in terms of $\Gamma(B)$.

Exchange matrices are an important ingredient in the definition of seeds.

\begin{defi}[Seeds]
  A \emph{seed} is a triple $\Sigma = (\x,\y,B)$ such that:
  \begin{enumerate}
    \item $\x=(x_1,\ldots,x_n)$ and $\y=(x_{n+1}, \ldots, x_{n+m})$, with $n$,~$m \ge 0$, are tuples consisting of altogether $n+m$ algebraically independent indeterminates over $K$;
      the set $\{ x_1,\ldots, x_{n+m} \}$ is called the \emph{cluster} of $\Sigma$.
    \item The elements in $\x$ are called \emph{exchangeable variables}; the elements in $\y$ are called \emph{frozen variables}.
    \item $B$ is an exchange matrix with $n+m$ rows and $n$ columns.
  \end{enumerate}
  Given a seed $\Sigma$, the field $\mathcal F = \mathcal F(\Sigma) = \mathbf{q}(K)(x_1,\ldots, x_{n+m})$ is called the associated \emph{ambient field}.
\end{defi}

We tacitly identify two seeds that arise from each other by a permutation of the exchangeable and frozen variables and a matching permutation of the rows and columns of $B$.

\begin{defi}[Acyclicity]
  A seed $\Sigma=(\x,\y,B)$  \textup{(}respectively, its exchange matrix $B$\textup{)} is called \emph{acyclic} if the ice quiver $\Gamma(B)$ is acyclic; that is, the full subquiver of $\Gamma(B)$ on the exchangeable vertices $[1,n]$ does not contain any oriented cycles.
\end{defi}

Note that the ice quiver of an acyclic seed may still have cycles involving frozen vertices, as is the case for the quiver in Example~\ref{exam:bigquiver}.

\begin{defi}[Principal coefficients] 
\label{def_principal_coeffs}
Let $\Sigma=(\x,\y,B)$ be a seed. We say that $\Sigma$ has \emph{principal coefficients} if $n=m$ and the $n\times n$ submatrix of $B$ formed by the last $n$ rows is the identity.
\end{defi}

We will consider the following polynomials.

\begin{defi}[Exchange polynomials] 
  Let $\Sigma=(\x,\y,B)$ be a seed with $\x=(x_1,\ldots,x_n)$ and $\y=(x_{n+1},\ldots,x_{n+m})$.
  Suppose that $i\in [1,n]$ is an exchangeable index. The polynomial 
  \[
    f_i=\prod_{\substack{j \in [1,n+m] \\ b_{ji}>0}} x_j^{b_{ji}} + \prod_{\substack{j \in [1,n+m] \\ b_{ji}<0}} x_j^{-b_{ji}}\in K[\x,\y]
  \]
  is called the \emph{exchange polynomial} of $x_i$ (with respect to the seed $\Sigma$).
\end{defi}

A crucial notion is the mutation of a seed.

\begin{defi}[Mutations of seeds]
  Let $\Sigma = (\x, \y, B)$ be a seed with $\x=(x_1,\ldots,x_n)$, with $\y=(x_{n+1},\ldots,x_{n+m})$, and with ambient field $\mathcal{F}$.
  For every $i \in[1,n]$, we define the \emph{mutation} of $\Sigma$ in the direction $i$ to be the seed $\mu_i(\Sigma) = (\x',\y',B')$ given by
  \begin{enumerate}
    \item $\x' = (x_1,\ldots,x_{i-1}, x_i', x_{i+1},\ldots, x_n)$ where $x_i'=\frac{f_i}{x_i}\in\mathcal{F}(\Sigma)$;
    \item $\y'=\y$;
    \item $B'=(b'_{kl}) \in \operatorname{Mat}_{(n+m)\times n}(\Z)$ with
      \[
        b'_{kl} = \begin{cases}
          - b_{kl} & \textrm{ if } k=i \textrm{ or } l=i; \\
          b_{kl} + \frac 12 (\lvert b_{ki}\rvert b_{il} + b_{ki}\lvert b_{il}\rvert) & \textrm{ otherwise.}
        \end{cases}
      \]
  \end{enumerate}
\end{defi}

It is easy to see that the mutation of a seed $\mu_i(\Sigma)$ is again a seed with the same number of exchangeable and frozen variables and the same ambient field.
Moreover, we have $(\mu_i\circ \mu_i)(\Sigma) =\Sigma$ for all seeds $\Sigma$ and all exchangeable indices $i$.
Thus, mutation induces the following equivalence relation.

\begin{defi}[Mutation equivalence]
  Given a seed $\Sigma=(\x,\y,B)$, its \emph{mutation class} is the set $\Mut(\Sigma)$ of all seeds which can be obtained from $\Sigma$ by applying successive mutations:
  \[
    \Mut(\Sigma) = \{\, \mu_{i_k} \circ \cdots \circ \mu_{i_2}\circ\mu_{i_1}(\Sigma) \ \vert \ k\geq 0 \text{ and } i_r\in[1,n]\text{ for all }r \,\}.
  \]
  Two seeds in the same mutation class are \emph{mutation-equivalent}.
\end{defi}

Denote by $\mathcal{X}$ the set of all exchangeable variables appearing in a seed in $\Mut(\Sigma)$. Now we are ready to state the definition of a cluster algebra.

\begin{defi}[Cluster algebras] \label{def:clusteralgebra}
  Let $\Sigma=(\x ,\y ,B)$ be a seed.
  The \emph{cluster algebra} associated to $\Sigma$ is the $K$-algebra
  \[
    A = A(\Sigma) = K[x, y \ \vert \ x \in \mathcal{X} ,\,  y ,  y^{-1} \in \y]\subseteq \mathcal{F}(\Sigma).
  \]
  The elements $x\in \mathcal{X}$ are called the \emph{cluster variables} of $A(\Sigma)$; the cluster variables in the initial seed $\Sigma$ are called \emph{initial cluster variables}; the elements $y\in\y$ are called the \emph{frozen variables} of $A(\Sigma)$. 
\end{defi}
Note that some authors do not invert the frozen variables in the definition of a cluster algebra. 
A cluster algebra is \emph{acyclic} if it has an acyclic seed $\Sigma$; however this does not imply that every seed of $A$ is acyclic.

There are several theorems in the literature that are helpful for us. Fomin--Zelevinsky's Laurent phenomenon \cite{FZ1} asserts that $A(\x,\y,B)\subseteq K[u^{\pm 1} \ \vert \ u\in\x \cup \y]$; the finite type classification \cite{FZ2} describes cluster algebras with only finitely many cluster variables by finite type root systems. Even if a cluster algebra admits infinitely many cluster variables it may be finitely generated as a $K$-algebra. More precisely, if the initial exchange matrix $B$ is acyclic, then the cluster algebra admits a nice presentation by finitely many generators and relations due to Berenstein--Fomin--Zelevinsky \cite{BFZ}.
\begin{theo}[Berenstein--Fomin--Zelevinsky] \label{Thm:BFZ}
  Let $\Sigma=(\x,\y,B)$ be a seed whose exchange matrix $B \in\operatorname{Mat}_{m+n,n}(\mathbb{Z})$ is acyclic. Let $\x = (x_1,\ldots,x_n)$, let $\y = (x_{n+1},\ldots,x_{n+m})$, and let $f_i$ with $i \in [1,n]$ be the exchange polynomials.
  Let $X_i$, $X_i'$ for $i \in [1,n]$, and $X_j$ for $j \in [n+1,n+m]$ denote algebraically independent indeterminates over $K$.
  Then the assignments
  \begin{align*}
    X_i\mapsto x_i \quad (i\in[1,n]), &&X_i'\mapsto \mu_{i}(\Sigma)_i \quad (i\in [1,n]),&&X_j \mapsto x_{j}\quad (j\in [n+1,n+m])
  \end{align*}
  induce an isomorphism of algebras
  \[
    K\big[ X_i,X_i',X_j^{\pm 1} \ \big\vert \  i\in[1,n],\, j\in[n+1,n+m] \big]/\langle X_i X_i' - f_i(X_1,\ldots, X_{n+m}) \ \vert \ i\in [1,n] \rangle \ \to\ A(\Sigma).
  \]
\end{theo}

See \cite[Corollary 1.21]{BFZ} for the previous result.
It follows from \cite[Corollary 1.17]{BFZ}, by which the lower bound of a cluster algebra with acyclic seed has a presentation as given above, and \cite[Theorem 1.20]{BFZ}, by which the lower bound of a cluster algebra with acyclic seed coincides with the cluster algebra itself.
We remark that the proof in \cite{BFZ} concerns the case $K=\Z$, but the isomorphism holds for $K$ a field of characteristic $0$ through base extension, as every field of characteristic $0$ is flat as $\Z$-algebra.

In particular, we note that the cluster algebra $A(\Sigma)$ is finitely generated and noetherian when $\Sigma$ is acyclic.
It is also known that $A(\Sigma)$ is integrally closed in this case, since it coincides with its upper cluster algebra (see \cite{M2}), which is integrally closed as an intersection of Laurent polynomial rings.

The presentation in Theorem~\ref{Thm:BFZ} also implies that, for an acyclic seed, the algebra obtained by inverting a subset of the initial cluster variables is isomorphic to the cluster algebra obtained by freezing the corresponding variables.
This will be a crucial ingredient for our inductive arguments in Section~\ref{sec:acyclic-prime}.

\begin{rema}
  Presentations of lower bound cluster algebras have been studied by Muller--Rajchgot--Zykoski in \cite{MRZ}.
  Locally acyclic cluster algebras, introduced by Muller in \cite{M}, are also finitely generated, noetherian, and integrally closed.
  These algebras are known to coincide with their upper bound algebras (see \cite{M2}), presentations of which have been studied by Matherne--Muller in \cite{MM}.
\end{rema}

\begin{rema} \label{Rem:FreezeTrivial}
  Let $\Sigma=(\x,\y,B)$ be a seed with $\x=(x_1,\ldots,x_n)$ and $\y = (x_{n+1}, \ldots, x_{n+m})$, and let $A=A(\Sigma)$.
  If $K$ is a field and $i \in [1,n]$ is isolated, then $x_ix_i'=2$ implies that $x_i$ is a unit in $A$.
  Thus, if we freeze $i$, we obtain an algebra isomorphic to the original one.
  Hence, if $K$ is a field,  we may without restriction assume that $[1,n]$ has no isolated exchangeable indices.
  This assumption ensures that our exchange polynomials are always non-units in the polynomial ring $K[\x,\y]$.
\end{rema}

From now on for the rest of the paper, unless otherwise stated, if $K$ is a field, we assume without restriction that a seed has no isolated exchangeable indices.
This will simplify the statement of many results, without restricting their generality in any way.

\subsection{Krull domains}\label{section commutative}

In this section, we summarize some basic properties of Krull domains, with a focus on their multiplicative ideal theory.
Our main references are   \cite{Fos,GHK}. While in \cite{GHK} the theory of \emph{Krull monoids} is developed, this encompasses Krull domains well.
(A domain $A$ is a Krull domain if and only if $A^\bullet$ is a Krull monoid.)

Let $A$ be a domain and $\mathbf{q}(A)$ its quotient field.

A \emph{fractional ideal} is an $A$-submodule $\fa \subseteq \mathbf{q}(A)$ such that there exists an $x \in \mathbf{q}(A)^\times$ with $x\fa \subseteq A$.
For a fractional ideal $\fa$ we define
\[
  \fa^{-1} = (A:\fa) = \{\, x \in \mathbf{q}(A) \ \vert\  x\fa \subseteq A \,\} \quad\text{and}\quad \fa_v = (\fa^{-1})^{-1}.
\]
A fractional ideal $\fa$ is \emph{divisorial} if $\fa = \fa_v$; it is \emph{invertible} if there exists a fractional ideal $\fb$ such that $\fa\fb=A$.
Not every fractional ideal is invertible, but if $\fa$ is invertible, then $\fa^{-1}$ as defined above is indeed the unique fractional ideal such that $\fa \fa^{-1} = A$.
In this case $\fa_v = \fa$, so that every invertible ideal is divisorial.
Every principal fractional ideal $xA$ with $x \in \mathbf{q}(A)^\times$ is invertible with inverse $x^{-1}A$.
In particular, principal fractional ideals are divisorial.

The divisorial closure satisfies that $(\fa_v)_v = \fa$, that $(xA)_v = xA$ for all $x \in \mathbf{q}(A)^\times$, and that $(\fa\fb)_v = (\fa_v \fb)_v = (\fa_v \fb_v)_v$ for all fractional ideals $\fa$,~$\fb$.
Moreover, if $\fa \subseteq \fb$, then $\fa_v \subseteq \fb_v$.

The domain $A$ is \emph{$v$-noetherian} (or a \emph{Mori domain}) if it satisfies the ascending chain condition on divisorial ideals.

An element $x \in \mathbf{q}(A)$ is \emph{almost integral} (over $A$) if there exists  $c \in \mathbf{q}(A)^\times$ such that $cx^n \in A$ for all $n \ge 0$.
We say that $A$ is \emph{completely integrally closed} if every almost integral element $x \in \mathbf{q}(A)$ belongs to $A$.
A noetherian domain is completely integrally closed if and only if it is integrally closed.

\begin{defi}[Krull domain]
  A \emph{Krull domain} is a domain $A$ that is $v$-noetherian and completely integrally closed.
\end{defi}

In particular, a noetherian domain is a Krull domain if and only if it is integrally closed.
Thus, (locally) acyclic cluster algebras are Krull domains.

Let $A$ be a Krull domain.
The following statements are equivalent for an ideal $\prim$ of $A$:
\begin{equivenumerate}
  \item $\prim$ is a height-$1$ prime ideal.
  \item $\prim$ is a nonzero divisorial prime ideal.
  \item $\prim$ is a maximal divisorial ideal (if $A$ is not equal to its quotient field).
\end{equivenumerate}

We write $\Xx (A)$ for the set of height-$1$ prime ideals and keep the following important property in mind.

\begin{lemm}
  If $A$ is a Krull domain and $a \in A^\bullet$, then the set $\{\, \prim \in \Xx (A) \ \vert\  a \in \prim \,\}$ is finite.
\end{lemm}

Let $\cI_v(A)$ denote the set of all divisorial ideals of $A$, let $\cF_v(A)$ denote the set of all divisorial fractional ideals, $\cI_v(A)^\bullet$ the set of nonzero divisorial ideals, and $\cF_v(A)^\times$ the set of nonzero divisorial fractional ideals.
On $\cF_v(A)$ we can define an associative operation by $\fa \cdot_v \fb := (\fa\fb)_v$.
The ideal $A$ is a neutral element for this operation, so that $\cF_v(A)$ is a monoid.
For fractional ideals $\fa_1$, $\ldots\,$, $\fa_k$ we define the notation
\[
  \vprod_{i=1}^k \fa_i = \Big(\prod_{i=1}^k \fa_i\Big)_v.
\]
By convention an empty product is equal to the trivial ideal $A$, and in particular $\fa^0 = (\fa^0)_v = A$.

\begin{theo}
  If $A$ is a Krull domain, then $(\cF_v(A)^\times, \cdot_v)$ is a free abelian group with basis $\Xx(A)$.
  Thus, every $\fa \in \cF_v(A)^\times$ has a representation as divisorial product
  \[
    \fa = \vprod_{\prim \in \Xx(A)} \prim^{n_\prim}.
  \]
  with uniquely determined $n_\prim \in \Z$, almost all of which are $0$.
  We have $\fa \in \cI_v(A)^\bullet$ if and only if $n_\prim \ge 0$ for all $\prim \in \Xx(A)$.
\end{theo}

\begin{proof}
  This follows from \cite[Corollary 3.14]{Fos}.
  Alternatively, this is shown in the more general setting of Krull monoids in \cite[Theorem 2.3.11]{GHK}.
  See \cite[Chapter 2.10]{GHK} for the connection between the monoid and the domain case.
\end{proof}

For $\fa \in \cF_v(A)^\times$ we define the \emph{$\prim$-adic valuation of $\fa$} as $\val_\prim(\fa) = n_\prim$ with $n_\prim$ as in the previous theorem.
For $\fa$,~$\fb \in \cI_v(A)^\bullet$, we have $\fa \subseteq \fb$ if and only if $\val_\prim(\fa) \ge \val_\prim(\fb)$ for all $\prim \in \Xx(A)$.

Since every principal fractional ideal $xA$ (with $x \in \mathbf{q}(A)^\times$) is divisorial and $(xA)\cdot_v(yA) = xyA$, the nonzero principal fractional ideals $\cH(A) = \{\, xA \mid x \in \mathbf{q}(A)^\times \,\} \subseteq \cF_v(A)^\times$ form a subgroup.

\begin{defi}[Class group]
  Let $A$ be a Krull domain.
  The \emph{\textup{(}divisor\textup{)} class group} of $A$ is $\cC(A) = \cF_v(A)^\times / \cH(A)$.
  We use additive notation for $\cC(A)$.
  If $\fa \in \cF_v(A)^\times$, we denote by $[\fa]$ its class in $\cC(A)$.
\end{defi}

An ideal $\fa \in \cI_v(A)$ is principal if and only if $[\fa]=0$.
The class group $\cC(A)$ and the subset $G_0 = \{\, [\prim] \in \cC(A) \mid \prim \in \Xx(A) \,\}$ of classes containing (divisorial, nonzero) prime ideals play a crucial role in the study of the arithmetic of $A^\bullet$.

Denote by $\Spec(A)$ the set of all prime ideals of $A$.
For an overring $A \subseteq B$ and an ideal $\mathfrak a \subseteq A$, we write $\mathfrak a B = \langle \mathfrak a \rangle_B$ for the extension of $\mathfrak a$ to $B$. Recall the following correspondence.

\begin{prop}
  \label{p-locprime}
  Let $A$ be a domain and $S \subseteq A^\bullet$ a multiplicatively closed set.
  There is an inclusion-preserving bijection
  \begin{align*}
    \{\, \prim \in \Spec (A) \mid \prim \cap S = \emptyset \,\} &\raw \Spec (S^{-1}A),\\
    \prim &\mapsto S^{-1}\prim = \prim (S^{-1}A),\\
    \qrim \cap A &\mapsfrom \qrim.
  \end{align*}
\end{prop}

Since the bijection is inclusion-preserving, it \emph{preserves the height} of any prime ideal $\fp$ of $A$ with $\fp \cap S = \emptyset$.
If $A$ is a Krull domain and $S \subseteq A^\bullet$ is a multiplicatively closed set, then the localization $S^{-1}A$ is a Krull domain as well.
If $\fa \in \cF_v(A)$, then $S^{-1}\fa \in \cF_v(S^{-1}A)$ and the following is a monoid homomorphism
\[
  j_S\colon (\cF_v(A), \cdot_v) \to (\cF_v(S^{-1}A),\cdot_{v}), \quad \mathfrak a \mapsto S^{-1}\mathfrak a .
\]
Moreover, $\Xx(S^{-1}A)$ is in bijection with $\{\, \prim \in \Xx(A) \mid \prim \cap S = \emptyset \,\}$.
If $\mathfrak a \in \cF_v(A)^\times$, then
\[
  S^{-1}\mathfrak a = \vprod_{\prim \in \Xx(A)} (S^{-1} \prim)^{\val_\prim(\mathfrak a)} = \vprod_{\substack{\prim \in \Xx(A)\\ \prim \cap S = \emptyset}} (S^{-1}\prim)^{\val_\prim(\mathfrak a)}.
\]
The divisorial product here is now taken in the ring $S^{-1}A$.
Thus, the factorization of $S^{-1}\mathfrak a$ arises from the one of $\mathfrak a$ by simply replacing every prime ideal $\prim$ having $\prim \cap S = \emptyset$  by $S^{-1}\prim$ and omitting all prime ideals that intersect $S$ non-trivially.

This idea will be useful for computing $\val_\prim(\mathfrak a)$ for divisorial ideals $\mathfrak a$ of cluster algebras: localizing by a cluster, we always get a Laurent polynomial ring, which is factorial, and easier to work in.
In this localization, we only lose the prime ideals containing a variable of the given cluster.
To figure out the exponents of these primes, we may localize by a different cluster.

The following theorem on the behavior of class groups under localization will be a key tool in this approach.

\begin{theo}[Nagata's Theorem, {\cite[\S7]{Fos}}] \label{thm:nagata}
  Let $A$ be a Krull domain and $S \subseteq A^\bullet$ a multiplicative set.
  Then $j_S$ induces an epimorphism $\overline{j_S} \colon \cC(A) \to \cC(S^{-1}A)$,
  and $\ker(\overline j_S)$ is generated by the classes of those $\prim \in \Xx(A)$ with $\prim \cap S \ne \emptyset$.
\end{theo}

In particular, if $S$ is generated by prime elements, then $\overline j_S$ is an isomorphism.
From this, one can deduce the following corollary.

\begin{coro} \label{c-nagata}
  Let $A$ be a domain and $S \subseteq A^\bullet$ a multiplicative set generated by prime elements.
  Then $A$ is factorial if and only if $S^{-1}A$ is.
\end{coro}

\subsection{Factorizations}

We recall some basic notions from factorization theory.
Let $A$ be a domain.
A non-unit $u \in A^\bullet$ is an \emph{atom} if $u=ab$ with $a$,~$b \in A^\bullet$ implies $a \in A^\times$ or $b \in A^\times$.
A non-unit $p \in A^\bullet$ is a \emph{prime element} if $p \mid ab$ with $a$,~$b \in A^\bullet$ implies $p \mid a$ or $p \mid b$; equivalently, the principal ideal $pA$ is a prime ideal in $A$.
Every prime element is an atom, but the converse is false in general.

Two elements $a$,~$b \in A^\bullet$ are \emph{associates} if there exists a unit $\varepsilon \in A^\times$ such that $a=b \varepsilon$; equivalently $aA = bA$.
The domain $A$ is \emph{atomic} if every non-unit $a \in A^\bullet$ is a product of atoms; it is \emph{factorial}  if every non-unit $a \in A^\bullet$ is a product of prime elements.
In the latter case, the factorization of $a$ into prime elements is unique up to order and associativity of the factors.
Factorial domains are often also called \emph{unique factorization domains}, or \emph{UFDs} in short.

Every $v$-noetherian domain, and hence every Krull domain, is atomic.
The following theorem shows the fundamental connection between factorial and Krull domains.
\begin{theo}
  Let $A$ be a domain.
  The following statements are equivalent.
  \begin{equivenumerate}
    \item \label{fact:fact} $A$ is factorial.
    \item \label{fact:atomsprime} $A$ is atomic and every atom of $A$ is a prime element.
    \item \label{fact:krull} $A$ is a Krull domain and $\cC(A)$ is trivial.
  \end{equivenumerate}
\end{theo}

\begin{proof}
  The equivalence \ref*{fact:fact}$\Leftrightarrow$\ref*{fact:atomsprime} is elementary, see \cite[Theorem 1.1.10]{GHK}.
 
  The equivalence \ref*{fact:fact}$\Leftrightarrow$\ref*{fact:krull} follows from \cite[Proposition 6.1]{Fos}.
  Alternatively, this is shown in the more general setting of (commutative, cancellative) monoids in \cite[Corollary 2.3.13]{GHK}.
  Again, see \cite[Chapter 2.10]{GHK} for the connection between monoids and domains.
\end{proof}

For cluster algebras, Geiss--Leclerc--Schr\"oer have shown the following.

\begin{theo}[\cite{GLS}] \label{Thm:ClusterAtoms}
  Let $\Sigma=(\x,\y,B)$ be a seed,  $\x = (x_1,\ldots, x_n)$, and $\y = (x_{n+1}, \ldots, x_{n+m})$.
  Let $A=A(\Sigma)$.
  \begin{equivenumerate}
    \item Any cluster variable is an atom.
    \item The group of units of $A$ is
      \[
        A^\times= K^\times \times \langle x_{j}^{\pm 1} \ \vert\ j \in [n+1,n+m] \rangle.
      \]
  \end{equivenumerate}
\end{theo}

Together with Nagata's Theorem, this implies a characterization of factorial cluster algebras, first observed by Lampe in \cite{L2}.

\begin{coro} \label{Cor:ClusterFactorial}
  Let $\Sigma=(\x,\y,B)$ be a seed with $\x = (x_1,\ldots, x_n)$ and $\y = (x_{n+1}, \ldots, x_{n+m})$, and let $A=A(\Sigma)$.
  Then the following statements are equivalent.
  \begin{equivenumerate}
    \item \label{Cor:ClusterFactorial:fact} $A$ is factorial,
    \item \label{Cor:ClusterFactorial:primeall} every cluster variable is a prime element,
    \item \label{Cor:ClusterFactorial:prime}every exchangeable variable $x_1$, $\ldots\,$,~$x_n$ of the seed $\Sigma$ is a prime element.
  \end{equivenumerate}
\end{coro}

\begin{proof}
  \ref*{Cor:ClusterFactorial:fact}${}\Rightarrow{}$\ref*{Cor:ClusterFactorial:primeall}:
  By Theorem~\ref{Thm:ClusterAtoms} every cluster variable is an atom.
  Since $A$ is factorial, every atom is a prime element.

  \ref*{Cor:ClusterFactorial:primeall}${}\Rightarrow{}$\ref*{Cor:ClusterFactorial:prime}: Clear.

  \ref*{Cor:ClusterFactorial:prime}${}\Rightarrow{}$\ref*{Cor:ClusterFactorial:fact}:
  By the Laurent phenomenon, the localization $A[x_1^{-1},\ldots,x_n^{-1}]$ is a Laurent polynomial ring over the factorial ring $K$, and hence itself factorial.
  The claim follows from Nagata's Theorem, in the form of Corollary~\ref{c-nagata}.
\end{proof}

\section{Exchange polynomials and partners}
\label{Section:exch-partners}

In this section we introduce the notion of partners, based on exchange polynomials having non-trivial common factors.
The concept of partners will play a central role in the computation of the class group of cluster algebras with an acyclic seed.

\subsection{Exchange polynomials}

Before defining partners, we take a closer look at the factorization of exchange polynomials.
For $d \ge 1$, we denote by $\mu_d (K)$ the group of $d$-th roots of unity in $K$, and by $\mu_d^*(K)$ the set of primitive $d$-th roots of unity.
Let $\overline \Q$ denote an algebraic closure of $\Q$.
We denote by
  \[
    \Phi_d(x,y) = \prod_{\zeta \in \mu^*_{d}(\overline \Q)} (x- \zeta y) \in \Z[x,y],
  \]
the homogenized $d$-th cyclotomic polynomial.
Recall that $\Phi_d(x,y)$, respectively the cyclotomic polynomial $\Phi_d(x,1)$, is irreducible over $\Z$ and hence also over $\Q$.
Over a given field $K$, the polynomial $\Phi_d(x,y)$ will either remain irreducible (if $\mu_d^*(K) = \emptyset$), or split into a product of $\varphi(d)$ linear factors (if $\mu_d^*(K) \ne \emptyset$); here $\varphi(d) = \card{(\Z/d\Z)^\times}$ denotes Euler's totient function.

A theorem of Ostrowski, \cite[Theorem IX]{O}, characterizes the absolute irreducibility of the exchange polynomials.
We need to make a few additional observations, and for this purpose restate the proof in full.

\begin{lemm} \label{lemma:poly}
  Let $D$ be a domain of characteristic $0$ and let $D[x_1,\ldots,x_k]$ be a polynomial ring.
  Let
  \[
    f = x_1^{u_1}\cdots x_k^{u_k} +  x_1^{v_1}\cdots x_k^{v_k} \in D[x_1, \ldots, x_k]  \setminus D \quad\text{with $u_iv_i = 0$ for all $i \in [1,k]$}.
  \]
 Set $d=\gcd(u_1,\ldots,u_k,v_1,\ldots,v_k)$, so that $f = g^d + h^d$ with monomials $g$,~$h \in D[x_1, \ldots, x_k]$ having disjoint support.
  Then the map
  \[
    D[x_1] \to D[x_1,\ldots,x_k],\ r \mapsto h^{\deg(r)} r(g/h)
  \]
  induces a bijection between factors of $x_1^d + 1$ and factors of $f$.
  In particular,
  \begin{enumerate}
    \item\label{lemma:poly:rep} all factors of $f \in D[x_1,\ldots,x_k]$ are contained in $D[g,h]$;
    \item\label{lemma:poly:bij} the factorization of $f$ into irreducible polynomials does not have any repeated factors;
    \item\label{lemma:poly:factalg} If $D=K$ is an algebraically closed field, then
      \[
        g^d + h^d = \frac{g^{2d} - h^{2d}}{g^d - h^d} = \prod_{\zeta \in \mu_{2d}(K) \setminus \mu_{d}(K)} (g - \zeta h),
      \]
      with $g - \zeta h$ irreducible in $K[x_1,\ldots, x_k]$.
    \item \label{lemma:poly:factZ} Suppose that $D=\Z$ or $D=\Q$. If $d=2^lc$ with $l \ge 0$ and $c$ odd, then
      \[
        g^d + h^d = \prod_{\substack{e \mid 2d \\ e \nmid d}} \Phi_e(g,h) = \prod_{e \mid c} \Phi_{2^{l+1} e}(g,h),
      \]
      with $\Phi_e(g,h)$ irreducible in  $D[x_1,\ldots, x_k]$.
  \end{enumerate}
\end{lemm}

\begin{proof}
  \cite[Theorem IX]{O} asserts that $f$ is absolutely irreducible if and only if $d = 1$.
  We follow Ostrowski's proof to obtain our slightly more refined claim.
  Since $f$ and $x_1^d + 1$ clearly do not have any non-trivial monomial factors, we may instead consider their factorizations in the the Laurent polynomial rings $D[x_1^{\pm 1}, \ldots, x_k^{\pm 1}]$, respectively $D[x_1^{\pm1 }]$.
  Here $f$ is associated to
  \[
    h^{-d}f = \prod_{j=1}^k x_j^{u_j - v_j} + 1.
  \]

Set $a_{1j} = (u_j - v_j )/d$ for $j \in [1,k]$.
  Since $\gcd(a_{1j}, \ldots, a_{kj}) = 1$, we can find $a_{ij} \in \Z$ for $i \in [2,k]$ and $j \in [1,k]$ such that the matrix $(a_{ij})$ is invertible over $\Z$.
  This induces an automorphism $\alpha$ of $D[x_1^{\pm 1}, \ldots, x_k^{\pm 1}]$ with
  \[
    \alpha(x_i) = \prod_{j=1}^k x_j^{a_{ij}}.
  \]
  Noting that $\alpha(x_1^d + 1) = h^{-d} f$, and that all factors of $x_1^d +1$ clearly lie in $D[x_1^{\pm 1}]$, the claim about the bijection between factors follows.

  The remaining claims are immediate consequences of this together with the corresponding well known statements about $x_1^d + 1$.
\end{proof}

\begin{defi} [Column-gcd] \label{defi:bgcd}
  Let $\Sigma=(\x,\y,B)$ be a seed with $\x=(x_1,\ldots,x_n)$ and $\y = (x_{n+1}, \ldots, x_{n+m})$.
  For $i \in [1,n]$, let
  \[
    \bgcd_i = \gcd(\, b_{ji} \ \vert\ j \in [1,n+m] \,) \in \Z_{\ge 0}
  \]
  denote the greatest common divisor of the $i$-th column of the exchange matrix.
\end{defi}

The characterization in \ref*{prop:exchpoly:factor}\ref*{prop:exchpoly:factor:3} in the following proposition also appears in \cite[Lemma 3.1]{BFZ}.

\begin{prop} \label{prop:exchpoly}
  Let $\Sigma=(\x,\y,B)$ be a seed with $\x=(x_1,\ldots,x_n)$ and $\y = (x_{n+1}, \ldots, x_{n+m})$.
  \begin{enumerate}
    \item \label{prop:exchpoly:rep} The exchange polynomials associated with $\Sigma$ have no repeated factors.
    \item \label{prop:exchpoly:factor} For non-constant exchange polynomials $f_i$ and $f_j$ with $i$,~$j \in [1,n]$, the following statements are equivalent.
      \begin{equivenumerate}
        \item \label{prop:exchpoly:factor:1} $f_i$ and $f_j$ have a non-trivial common factor;
        \item \label{prop:exchpoly:factor:2} $f_i = g^{\bgcd_i} + h^{\bgcd_i}$ and $f_j = g^{\bgcd_j} + h^{\bgcd_j}$ with monomials $g$,~$h \in K[\x,\y]$ having disjoint support, at least one of which is non-trivial, and $\val_2(\bgcd_i) = \val_2(\bgcd_j)$.
        \item \label{prop:exchpoly:factor:3} There exist odd $d_i$,~$d_j \in \Z$ such that $d_j b_{ki} = d_i b_{kj}$ for all $k \in[1,n+m]$.
      \end{equivenumerate}
      In this case, let $l=\val_2(\bgcd_i)=\val_2(\bgcd_j)$ and let $c=\gcd(\bgcd_i , \bgcd_j) / 2^l$ be the greatest common divisor of the odd parts of $\bgcd_i$ and $\bgcd_j$.
      Then
      \[
        g^{2^l c} + h^{2^l c} \in \Z[\x,\y]
      \]
      is a greatest common divisor of $f_i$ and $f_j$.
    \item \label{prop:exchpoly:equiv}  Having a non-trivial common factor is an equivalence relation on the set of exchange polynomials $\{ f_1, \ldots, f_n \}$.
    \item \label{prop:exchpoly:irredKbar}   For $K$ an algebraically closed field,  a non-constant exchange polynomial $f_i$ is irreducible if and only if $\bgcd_i=1$.
    \item \label{prop:exchpoly:irred}   For $K=\Z$ or $K=\Q$, a non-constant exchange polynomial $f_i$ is irreducible if and only if $\bgcd_i$ is a power of $2$.
  \end{enumerate}
\end{prop}

\begin{proof}
  \ref*{prop:exchpoly:rep}
  If the exchange polynomial $f_i$ is constant, then $i$ is an isolated vertex and we assumed this was not the case when $K$ is a field. Thus, if $f_i$ is constant, then $K=\Z$ and $f_i=2$, where the claim is clear.
  If $f_i$ is non-constant, the claim follows from \ref{lemma:poly:bij} of Lemma~\ref{lemma:poly}.

  \ref*{prop:exchpoly:factor}
  \ref*{prop:exchpoly:factor:1}${}\Rightarrow{}$\ref*{prop:exchpoly:factor:2}
  We may without restriction assume that $K$ is an algebraically closed field.
  For $k \in \{i,j\}$ write the exchange polynomial $f_k$ as
  \[
    f_k = g_k^{\bgcd_k} + h_k^{\bgcd_k} \quad\text{with $g_k$, $h_k \in K[\x,\y]$ monomials with disjoint support.}
  \]
   Using \ref{lemma:poly:factalg} of Lemma~\ref{lemma:poly}, any irreducible factor of $f_k$ is asssociate to
  \[
    g_k  - \zeta_k h_k
  \]
  for some $\zeta_k \in \mu_{2\bgcd_k}(K) \setminus \mu_{\bgcd_k}(K)$.
  Observe that $\val_2(\ord(\zeta_k)) = \val_2(\bgcd_k) + 1$, and hence that $\ord(\zeta_k)$ determines $\val_2(\bgcd_k)$.

  Now assume that $f_i$ and $f_j$ share a non-trivial common factor $h$.
  Without restriction we may assume that that $h$ is irreducible.
  Possibly replacing $h$ by an associate,
  \[
    h = g_i  - \zeta_i h_i = \lambda ( g_j - \zeta_j h_j)
  \]
  with $\lambda \in K^\times$, with $\zeta_i \in \mu_{2\bgcd_i}(K) \setminus \mu_{\bgcd_i}(K)$, and with $\zeta_j \in \mu_{2\bgcd_j}(K) \setminus \mu_{\bgcd_j}(K)$.
  We must have $\{g_i, h_i\} = \{g_j, h_j\}$.
  Moreover, either $\lambda = 1$, or otherwise $\lambda = -\zeta_j^{-1}$.
  In either case, it follows that $\val_2(\bgcd_i) = \val_2(\bgcd_j)$ since $\ord(\zeta_i) = \ord(\zeta_j) = \ord(\zeta_j^{-1})$.

  \ref*{prop:exchpoly:factor:2}${}\Rightarrow{}$\ref*{prop:exchpoly:factor:3}
  Let $l = \val_2(\bgcd_i) = \val_2(\bgcd_j)$.
  Setting $d_i = \bgcd_i / 2^l$ and $d_j = \bgcd_j / 2^l$, we have either $d_j b_{ki} = d_i b_{kj}$ for all $k \in [1,n+m]$ or $d_j b_{ki} = -d_i b_{kj}$ for all $k \in [1,n+m]$.
  
  \ref*{prop:exchpoly:factor:3}${}\Rightarrow{}$\ref*{prop:exchpoly:factor:1}
If $d_i$,~$d_j \in \Z$ are odd integers such that $d_j b_{ki} = d_i b_{kj}$ for all $k \in[1,n+m]$, then the odd integers $d_i'=d_i/\gcd(d_i,d_j)$ and $d_j'=d_j/\gcd(d_i,d_j)$ satisfy the same relations again. Hence without loss of generality we may assume $\gcd(d_i,d_j)=1$. Then $b_{ki} / d_i$ for $k \in [1,n+m]$ is an integer and we can write
  \[
    f_i = \Big(\prod_{\substack{k \in [1,n+m] \\ b_{ki} > 0}} x_k^{b_{ki}/d_i}\Big)^{d_i} + \Big( \prod_{\substack{k \in [1,n+m] \\ b_{ki} < 0}} x_k^{-b_{ki}/d_i} \Big)^{d_i}.
  \]
  Because $d_i$ is odd, this implies that
  \[
    \prod_{\substack{k \in [1,n+m] \\ b_{ki} > 0}} x_k^{b_{ki}/d_i} + \prod_{\substack{k \in [1,n+m] \\ b_{ki} < 0}} x_k^{-b_{ki}/d_i}
  \]
  is a divisor of $f_i$.
  Using $b_{ki} / d_i= b_{kj} / d_j$, we see that $f_j$ has the same divisor.

  Finally, that the stated polynomial is a greatest common divisor of $f_i$ and $f_j$ follows from \ref{lemma:poly:factZ} of Lemma~\ref{lemma:poly}.

  \ref*{prop:exchpoly:equiv}
  Reflexivity and symmetry are clear.
  The only constant exchange polynomial that can appear is $2$.
  For non-constant exchange polynomials, the transitivity follows from \ref*{prop:exchpoly:factor}\ref*{prop:exchpoly:factor:3}.

  \ref*{prop:exchpoly:irredKbar}
  By Lemma~\ref{lemma:poly}.
  
  \ref*{prop:exchpoly:irred}
  By \ref{lemma:poly:factZ} of Lemma~\ref{lemma:poly}.
\end{proof}

As non-constant exchange polynomials are primitive, a non-constant exchange polynomial is a prime element of $K[\x,\y]$ if and only if it is an irreducible polynomial.

\begin{coro} \label{coro:parallel-irred}
  Let $Q$ be an acyclic ice quiver without parallel arrows.
  Then every exchange polynomial of the seed $\Sigma=(\x,\y, Q)$ is a prime element in $K[\x,\y]$.
\end{coro}

\begin{proof}
  If $i \in [1,n]$ is an isolated vertex and $K=\Z$, then $f_i = 2$ is a prime element.
  If $K$ is a field, our standing assumption is that $Q$ does not contain any isolated vertices.
  Hence, in any other case,
  \[
    f_i = \prod_{j \in N_-(i)} x_j + \prod_{j \in N_+(i)} x_j
  \]
  is a non-constant polynomial with $\bgcd_i = 1$.
  Hence $f_i$ is irreducible by \ref{prop:exchpoly:irredKbar} of Proposition~\ref{prop:exchpoly}.
\end{proof}

\begin{coro} \label{coro:principal_irred}
  Suppose that the seed $\Sigma=(\x,\y, Q)$ has principal coefficients. Then every exchange polynomial is a prime element in $K[\x,\y]$ and two distinct exchange polynomials are coprime.
\end{coro}
\begin{proof}
By construction we have $\bgcd_i = 1$ for all $i\in[1,n]$, hence every $f_i$ is irreducible by \ref{prop:exchpoly:irredKbar} of Proposition~\ref{prop:exchpoly}. Moreover, for $i\neq j$ there do not exist odd $d_i$,~$d_j \in \Z$ such that $d_j b_{ki} = d_i b_{kj}$ for all $k \in[1,n+m]$, hence $f_i$ and $f_j$ do not have a non-trivial common factor by \ref{prop:exchpoly:factor}\ref{prop:exchpoly:factor:3} of Proposition~\ref{prop:exchpoly}.
\end{proof}

Since irreducible factors of non-constant exchange polynomials correspond to irreducible factors of $x^d +1$ for some $d \ge1$, it is now easy to count the number of irreducible factors in an exchange polynomial.
For instance, if $K$ is an algebraically closed field, then $x^d + 1$ splits into linear factors, and therefore has exactly $d$ irreducible factors.
On the other hand, if $K=\Z$ or $K=\Q$, and $d=2^lc$ with $l \ge 0$ and $c$ odd, then the irreducible factors of $x^d+1$ are in bijection with the divisors of $c$.
Hence, $x^d+1$ is a product of $\sigma_0(c)$ irreducible polynomials, where $\sigma_0(c)$ denotes the number of positive divisors of $c \in \Z_{>0}$.
For other fields of characteristic $0$, the number of factors can be determined in terms of the primitive roots of unity contained in $K$.

Since a greatest common divisor of two exchange polynomials has a similar form as the exchange polynomials themselves, the discussion in the previous paragraph applies to the counting of common irreducible factors of two or more exchange polynomials as well.
Thus, the counting of (common) irreducible factors of two or more exchange polynomials boils down entirely to elementary divisibility properties of the $\bgcd_i$ and knowledge about which roots of unity are contained in $K$.
We come back to this in Section~\ref{sec:class-group-acyclic}, where this information leads us directly to the determination of the rank of the class group of a cluster algebra with acyclic seed.

\subsection{Partner sets}\label{subsection partners}

Fix a seed $\Sigma=(\x,\y,B)$ with $\x =(x_1,\ldots, x_n)$ and $\y=(x_{n+1}, \ldots, x_{n+m})$. Let $A$ be the cluster algebra $A(\Sigma)$.

\begin{defi}[Partners] \label{defi:partners}
  \mbox{}
  \begin{enumerate}
    \item
      Two exchangeable indices $i$,~$j \in [1,n]$ are \emph{partners} if the exchange polynomials $f_i$ and $f_j$ have a non-trivial common factor in $K[\x,\y]$.
      Partnership is an equivalence relation on the set of exchangeable indices by Proposition~\ref{prop:exchpoly}.
      An equivalence class is called a \emph{partner set}.

    \item
      Let $g \in K[\x,\y]$ be a non-unit.
      An exchangeable index $i \in [1,n]$ is called a \emph{$g$-partner} \textup{(}or a \emph{$gK[\x,\y]$-partner}\textup{)} if the corresponding exchange polynomial $f_i$ is divisible by $g$.
      The \emph{set of $g$-partners} is
      \[
        \ptns{g}=\ptns{gK[\x,\y]} = \{\, i \in [1,n] \mid \text{$i$ is a $g$-partner} \,\}.
      \]
  \end{enumerate}
\end{defi}

It is convenient to define the set of $g$-partners in terms of the principal ideal $gK[\x,\y]$: if $gK[\x,\y]=g'K[\x,\y]$, then $\ptns{g}=\ptns{g'}$.

  If $i$,~$j \in [1,n]$ are both $g$-partners for some $g$, then trivially they are also partners.
  By our discussion on exchange polynomials, Proposition~\ref{prop:exchpoly}, the partner relation between exchangeable indices can be described entirely in terms of the exchange matrix $B$.
  Specifically, for $i \in [1,n]$, let $b_{*i}$ denote the $i$-th column of the exchange matrix $B$ and $\bgcd_i$ the greatest common divisor of its entries.
  By \ref{prop:exchpoly:factor}\ref{prop:exchpoly:factor:3} of Proposition~\ref{prop:exchpoly}, two non-isolated exchangeable indices $i$ and $j$ are partners if and only if $\val_2(\bgcd_i) = \val_2(\bgcd_j)$ and $b_{*i}/\bgcd_i = \pm b_{*j}/\bgcd_j$.
  Isolated vertices (which we only allow in the case $K=\Z$) are always partners, since their exchange polynomial is the constant $2$.

  If two indices $i$ and $j$ are partners, then
\begin{itemize}
\item either $N_{-}(i)=N_{-}(j)$ and $N_{+}(i)=N_{+}(j)$ hold, or $N_{-}(i)=N_{+}(j)$ and $N_{+}(i)=N_{-}(j)$ hold; in short
  \[
    \{ N_-(i), N_+(i) \} = \{ N_-(j), N_+(j) \},
  \]
  and
  \item we have $\val_2(\bgcd_i) = \val_2(\bgcd_j)$.
\end{itemize}
In particular, we have $N(i)=N(j)$ for partners $i$,~$j$.

If the exchange polynomials are all prime elements, then $i$,~$j$ are partners if and only if $f_i=f_j$.
This is the case when $B=B(Q)$ for $Q$ an ice quiver without parallel arrows.
In that case, also $\bgcd_i\in \{0,1\}$ for all $i \in [1,n]$.
Thus, $i$ and $j$ are partners if and only if $\{ N_-(i), N_+(i) \} = \{ N_-(j), N_+(j) \}$.

\begin{lemm}
  Two partners are never neighbors of each other.
\end{lemm}

\begin{proof}
  Let $i$,~$j \in [1,n]$ be partners and suppose they were neighbors. Then $j \in N(i) = N(j)$, contradicting that $\Gamma(B)$ has no loops.
\end{proof}

\begin{exam}
  \label{exam:bigquiver}
  Consider the following ice quiver, where the frozen vertices are boxed. The number over an arrow indicates that there are that many parallel arrows. Let $A$ be the cluster algebra over $\Q$ defined by $B(Q)$.
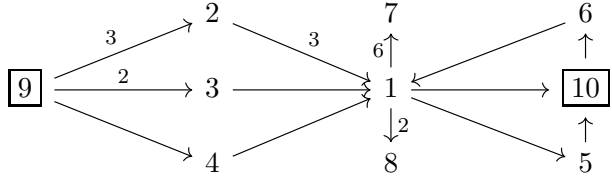
  \begin{center}
    \begin{tikzcd}[row sep = small]
      && 2 \arrow[drr, "3"]  
      & & 7  && 6 \arrow[dll]  \\
      \boxed{9} \arrow[rr, "2"]\arrow[urr, "3"] \arrow[drr] && 3  \arrow[rr]
      & & 1\arrow[rr] \arrow[drr] \arrow[d, "2"] \arrow[u, "6"] & & \boxed{10}\arrow[u]\\
    && 4 \arrow[urr] && 8  && 5 \arrow[u]\end{tikzcd}
  \end{center}
  
The exchange polynomials in this case are:
\[ f_1= x_2^3 x_3 x_4 x_6 + x_5 x_7^6 x_8^2 x_{10} \ \ ; \ \ f_2= x_9^3 + x_1^3 \ \ ; \ \ f_3= x_9^2 + x_1  \ \ ; \ \ f_4=x_9 + x_1\]\[f_5 = f_6 = x_1 + x_{10} \ \ ; \ \ f_7= x_1^6+1 \ \ ; \ \ f_8= x_1^2 +1. \]  
So $f_2$ and $f_4$ have the irreducible factor $x_9 + x_1$ in common. The polynomials $f_7$ and $f_8$ share the factor $x_1^2+1$. The partner sets are $P(x_1 + x_9) = \{ 2,4\}$, $P(x_1 + x_{10})= \{ 5,6\}$, and $P(x_1^2 +1) = \{ 7,8\}$, plus the unitary sets $\{ 1 \}$ and $\{3\}$. 
\end{exam}

For every exchangeable index $i\in [1,n]$ we denote by $x_i'$ the cluster variable obtained by mutation of $\Sigma$ in direction $i$.
Equivalently, $x_i'$ is the generator in Berenstein--Fomin--Zelevinsky's presentation of $A(\Sigma)$.

\begin{lemm}
\label{Lemma:MutatedPartners}
Suppose that $i_1$, $\ldots\,$,~$i_r \in [1,n]$ are partners in $\Sigma$ with corresponding cluster variables $x_{i_1}$, $\ldots\,$,~$x_{i_r}$.
Then $(\mu_{i_s}\circ\mu_{i_t})(\Sigma)=(\mu_{i_t}\circ\mu_{i_s})(\Sigma)$ for all $s$,~$t \in [1,r]$. In particular, the seed $(\mu_{i_r}\circ\dots\circ\mu_{i_1})(\Sigma)$ contains the cluster variables $x_{i_1}'$, $\ldots\,$,~$x_{i_r}'$. 
\end{lemm}

\begin{proof}
  The claim follows from a more general statement, compare Fomin--Williams--Zelevinsky in \cite[Exercise 2.7.7]{FWZ}:
  If two indices $i$,~$j \in [1,n]$ are not neighbors of each other, then $\mu_i\circ\mu_j=\mu_j\circ\mu_i$.\end{proof}

\begin{rema}
  A result of Caldero and Keller (see \cite[Corollary 4]{CK}) asserts that if two ice quivers $Q$ and $\widetilde{Q}$ are acyclic and the two seeds $(\x,\y,B(Q))$ and $(\widetilde{\x},\widetilde{\y},B(\widetilde{Q}))$ are mutation-equivalent, then $Q$ and $\widetilde{Q}$ are related by a sequence of mutations at sinks or sources. These mutations do not change the underlying undirected diagram of the quiver. The theorem of Caldero--Keller implies that for every $k \in \Z_{>0}$ the number of partner sets of size $k$ in $Q$ is equal to the number of partner sets of size $k$ in $\widetilde{Q}$.
\end{rema}

\section{Class groups}

\label{sec:class-groups}
In this section we prove our first main theorem, Theorem~\ref{Thm:MainGeneral}.
The results in this section apply to all cluster algebras that are Krull domains.
We show that a cluster algebra that is a Krull domain always has a finitely generated free abelian class group, with rank determined by the number of height-$1$ primes over the exchangeable variables of a fixed cluster.
This proves a conjecture of Lampe, who, in \cite{L2}, conjectured that the class group of every acyclic cluster algebra is torsion-free.
Moreover, we establish that every class contains infinitely many height-$1$ primes.

\subsection{Class groups in Krull domains}\label{subsection class group Krull domains}

To begin with, we work in a slightly more general setting.

\begin{theo} 
\label{Thm:CG}
Let $A$ be a Krull domain, and let $x_1$, $\ldots\,$,~$x_n\in A$ be elements such that the localization $A[x_1^{-1},\ldots,x_{n}^{-1}] = D[x_1^{\pm 1},\ldots,x_n^{\pm 1}]$ is a Laurent polynomial ring in the indeterminates $x_1$, $\ldots\,$,~$x_n$ over some factorial subring $D \subseteq A$. Let $\prim_1$, $\ldots\,$,~$\prim_t$ be the pairwise distinct height-$1$ prime ideals of $A$ containing one of the elements $x_1$, $\ldots\,$,~$x_n$. Suppose that
\[
x_iA=\vprod_{j=1}^t\prim_j^{a_{ij}}
\]
with $\mathbf{a}_i=(a_{ij})_{j=1}^t\in\Z_{\ge 0}^t$.
Then $\mathcal{C}(A)\cong \mathbb{Z}^t/ \langle \mathbf{a}_i\ \vert \ i\in [1,n] \rangle$.
\end{theo}

\begin{proof}
  Let $H \subseteq \mathcal{C}(A)$ denote the subgroup of $\mathcal C(A)$ generated by the classes of height-$1$ prime ideals that contain one of $x_1$, $\ldots\,$,~$x_n$.
  In other words, $H$ is generated by $[\prim_1]$, $\ldots\,$,~$[\prim_t]$.
  Nagata's Theorem implies that there is a short exact sequence
\[
0\longrightarrow H \longrightarrow \mathcal{C}(A) \longrightarrow \mathcal{C}(A[x_1^{-1},\ldots,x_n^{-1}])\longrightarrow 0.
\]
By assumption we have $\mathcal{C}(A[x_1^{-1}, \ldots, x_n^{-1}])  = \mathcal{C}(D[x_1^{\pm 1}, \ldots, x_n^{\pm 1}])\cong  \mathcal{C}(D)=0$.
It follows that $H \cong \mathcal{C}(A)$ and thus, in particular, that $\mathcal C(A)$ is generated by $[\prim_1]$, $\ldots\,$,~$[\prim_t]$.

Now suppose that there are integers $m_j\in \mathbb{Z}$ with $j\in[1,t]$ such that $\sum_{j=1}^tm_j[\prim_j]=0$ in $\mathcal{C}(A)$.
Then the corresponding divisorial product of height-$1$ prime ideals is principal, that is, there exists an element $f\in \mathbf q(A)$ such that
\begin{equation}
  \label{Eqn:Dec1}
  f A=\vprod_{j=1}^t \prim_j^{m_j}.
\end{equation}
Localization at the multiplicative set $S$ generated by $x_1$, $\ldots\,$,~$x_n$ yields
\[
fA[x_1^{-1},\ldots,x_n^{-1}]=\vprod_{\substack{j=1 \\ \prim_j\cap S= \emptyset}}^t (S^{-1}\prim_j)^{m_j} = A[x_1^{-1},\ldots,x_n^{-1}].
\]
Hence $f$ lies in $A[x_1^{-1},\ldots,x_{n}^{-1}]^{\times} = D[x_1^{\pm 1},\ldots,x_n^{\pm 1}]^{\times}$.
Since $D[x_1^{\pm 1},\ldots,x_n^{\pm 1}]$ is a Laurent polynomial ring over $D$, we may thus write $f=\lambda x_1^{r_1}x_2^{r_2}\cdots x_n^{r_n}$ with some $\lambda\in D^{\times}\subseteq A^{\times}$ and integers $r_i\in\mathbb{Z}$. Substituting the factorizations of $x_iA$ as divisorial products of height-$1$ primes, we obtain
\begin{equation}
\label{Eqn:Dec2}
fA=\vprod_{i=1}^n\Big(\vprod_{j=1}^t\prim_j^{a_{ij}}\Big)^{r_i}=\vprod_{j=1}^t \prim_j^{\sum_{i=1}^nr_ia_{ij}}.
\end{equation}
A comparison of Equations \eqref{Eqn:Dec1} and \eqref{Eqn:Dec2} yields $m_j=\sum_{i=1}^nr_ia_{ij}$ for all $j\in [1,t]$.
In other words $\mathbf{m}=(m_j)_{j=1}^t=\sum_{i=1}^nr_i\mathbf{a}_i$.
This implies $\mathbf{m}\in\langle \mathbf{a}_i \ \vert \ i\in [1,n]\rangle\subseteq\mathbb{Z}^t$ which finishes the proof of the theorem.
\end{proof}

In the next proof we use that every Krull domain possesses the \emph{approximation property}:
Let $\prim_1$, $\ldots\,$,~$\prim_t$ be pairwise distinct height-$1$ prime ideals and $e_1$, $\ldots\,$,~$e_t \in \Z$.
Then there exists $a \in K^\times$ such that $\val_{\prim_j}(a) = e_j$ for $j \in [1,t]$ and $\val_\fp(a) \ge 0$ for every other height-$1$ prime ideal $\prim$.
See \cite[Theorem 5.8]{Fos} or \cite[Definition 2.5.3 and Corollary 2.10.10]{GHK}.

\begin{theo} \label{Thm:PrimeDistribution}
  We keep the assumptions of Theorem~\ref{Thm:CG} and moreover assume that $D$ is an infinite domain, and that either $n\ge 2$, or $n=1$ and $D$ has at least $\card{D}$ height-$1$ prime ideals.
  Then every class of $\mathcal C(A)$ contains precisely $\card{D}$ height-$1$ prime ideals.
\end{theo}

\begin{proof}
  We first show the following claim.
 
  \begin{claim}[A]
    For every $j \in [1,t]$, let $e_j \in \Z_{\ge 0}$.
    Then there exist at least $\card{D}$ height-$1$ prime ideals $\qrim \subseteq A$ with $[\qrim] = \sum_{j=1}^t -e_j [\prim_j]$.
\end{claim}

\begin{proof}
  By the approximation property there exists an element $a \in A$ such that $\val_{\prim_j}(a) = e_j$ for all $j \in [1,t]$.

  Write
  \[
    a = \sum_{i=r}^s x_1^i a_i
  \]
  with $r \le s$, with $a_i \in D[x_2^{\pm 1}, \ldots, x_n^{\pm 1}]$, with $a_r \ne 0$, and with $a_s \ne 0$.
  Let $p \in D[x_2,\ldots, x_n]$ be a prime element with $p \not\in \prim_j$ for all $j \in [1,t]$ and such that $p$ does not divide $a_r$ in $D[x_2^{\pm 1}, \ldots, x_n^{\pm 1}]$.
  By our assumptions on $D$ and $n$, there exist at least $\card{D}$ pairwise non-associated such prime elements.
  
  Choose $N \in \Z_{>0}$ such that $N > e_j$ for all $j \in [1,t]$ and $N > s$.
  Consider
  \[
    b = (x_1 \cdots x_n)^N + p a \in A.
  \]
  Then $\val_{\prim_j}(b) = \val_{\prim_j}(a) = e_j$ for all $j \in [1,t]$.
  By the Eisenstein criterion, $x_1^{-r} b$ is irreducible as a polynomial in $x_1$ over $D[x_2^{\pm 1}, \ldots, x_n^{\pm 1}]$.
  Since its leading coefficient, $(x_2 \cdots x_n)^N$, is a unit in $D[x_2^{\pm 1}, \ldots, x_n^{\pm 1}]$, it also does not have any non-unit constant factor.
  We conclude that $b$ is a prime element of $D[x_1^{\pm 1}, \ldots, x_n^{\pm 1}] = A[x_1^{-1},\ldots, x_n^{-1}]$.

  Thus, by Proposition~\ref{p-locprime},
  \[
    bA = \mathfrak q \cdot_v \vprod_{j=1}^t \prim_j^{e_j}
  \]
  with a height-$1$ prime ideal $\mathfrak q$ of $A$ that does not contain any of $x_1$, $\ldots\,$,~$x_n$, and $[\mathfrak q] = - \sum_{j=1}^t e_j [\mathfrak p_j]$.
  Varying the choice of $p$, we find at least $\card{D}$ different such ideals.
  \renewcommand{\qedsymbol}{\ensuremath{\square} (Claim \textbf{A})}
  \end{proof}

  Let now $g \in \mathcal C(A)$ be an arbitrary class.
  Then
  \begin{equation} \label{eq:gclass}
    g = \sum_{j=1}^t c_j \mathfrak [\prim_j] \quad \text{with $c_j \in \Z$ for $j \in [1,t]$}.
  \end{equation}
  Considering some $x_l A$ with $x_l \in \prim_j$, we have
  \[
    [\prim_j] = a_{lj} [\prim_j] -(a_{lj}-1) [\prim_j] = \sum_{\substack{k =1 \\ k\ne j}}^t - a_{lk} [\prim_k] - (a_{lj} - 1) [\prim_j].
  \]
  Thus, if $c_j > 0$, we can replace $[\prim_j]$ by a negative linear combination of $[\prim_k]$'s.
  Hence, without restriction,  we can assume $c_j \le 0$ for all $j \in [1,t]$ in Equation~\eqref{eq:gclass}.
  Therefore there exist at least $\card{D}$ height-$1$ prime ideals $\qrim$ in $A$ with $[\qrim] = g$ by Claim~\textbf{A}.

  Noting that $\card{A} = \card{D[x_1,\ldots,x_n]} = \card{D}$, and that in a Krull domain every height-$1$ prime ideal is finitely generated as a divisorial ideal, it is clear that $A$ has at most $\card{D}$ height-$1$ prime ideals.
\end{proof}

\subsection{Class groups for cluster algebras}\label{subsection class groups cluster algebras}

In this subsection we apply the previous results in the particular case where the Krull domain $A$ is also a cluster algebra.

\begin{lemm} \label{Lemma:IndicatorPrimes}
  Let $\Sigma=(\x,\y,B)$ be a seed with $\x=(x_1, \ldots, x_n)$, and $A= A(\Sigma)$.
  Suppose that $A$ is a Krull domain.
  Then, for each $i \in [1,n]$, there exists a height-$1$ prime ideal $\prim \subseteq A$ such that $\val_{\prim}(x_i) = 1$ and $\val_{\prim}(x_j) = 0$ for $j \in [1,n] \setminus \{i\}$.
\end{lemm}

\begin{proof}
  We mutate $\Sigma$ in direction $i$ to obtain a new seed $\x'$ in which $x_i$ is replaced by $x_i'$.
  Since $x_i$ is not contained in the support of the exchange polynomial $f_i$, we have $f_i \in K[\x',\y]$.
  Let $A' = A[u^{-1} \ \vert \ u \in \x' \cup \y]$.
  Let $p \in K[\x',\y]$ be a prime factor of $f_i$ and set $\prim = pA' \cap A$.
  By Krull's Principal Ideal Theorem, the prime ideal $p A'$ has height $1$, and therefore so does $\prim$.
  For $j \in [1,n] \setminus \{i\}$ we have $x_j \in (A')^\times$, and therefore $x_j \not \in \prim$.
  Hence $\val_\prim(x_j) = 0$.
  Because $f_i$ does not have any repeated factors, we have $\val_{\prim}(x_i) = \val_{\prim}(f_i) = 1$.
\end{proof}

We can now prove the first main theorem of the paper.

\begin{proof}[Proof of Theorem~\ref{Thm:MainGeneral}] \label{Proof:MainGeneral}
  We apply Theorem~\ref{Thm:CG} to the cluster algebra $A$ and the set $\{x_1,\ldots, x_{n+m}\}$.
  In this situation $A[x_1^{-1}, \ldots, x_n^{-1}]=A[x_1^{-1}, \ldots, x_{n+m}^{-1}]$ is equal to the Laurent polynomial ring $K[x_1^{\pm 1},\ldots, x_{n+m}^{\pm 1}]$, which is factorial.
    Thus, we know that $\mathcal C(A)$ is generated by $t$ elements, with $n$ relations stemming from $x_1$, $\ldots\,$,~$x_n$ (the frozen variables $x_{n+1}$, $\ldots\,$,~$x_{n+m}$ are units in $A$ and hence not contained in any prime ideal; the relations they give are trivial).
    Let $\mathbf a=(a_{ij})_{i \in [1,n], j \in [1,t]}$ denote the matrix of relations, as in Theorem~\ref{Thm:CG}, omitting the zero rows obtained from the frozen variables.
    
  By Lemma~\ref{Lemma:IndicatorPrimes}, for each row $i\in [1,n]$ of $\mathbf a$ we can find an index $j \in [1,t]$ with $a_{ij}=1$ and $a_{kj}=0$ for $k \in [1,n] \setminus \{i\}$.
  This implies that we can use the $i$-th row to eliminate the $j$-th generator.
  The additional property that $a_{kj}=0$ for $k \in [1,n] \setminus \{i\}$, implies that any other cluster variable will allow us to eliminate a different generator.
  Thus, we find that $\mathcal C(A)$ is a free abelian group of rank $t-n$.

  If $n+m \ge 2$, or $n+m = 1$ and $K=\Z$, we can apply Theorem~\ref{Thm:PrimeDistribution} to obtain that every class contains exactly $\card{K}$ height-$1$ prime ideals.
  Suppose that $n+m=1$ and that $K$ is a field.
  By our standing assumption, there is no isolated exchangeable index, so that necessarily $n=0$ and $m=1$.
  But then $A=K[x_1^{\pm 1}]$ is a Laurent polynomial ring, $\cC(A) = 0$, and $A$ contains $\card{K}$ pairwise non-associated prime elements.
\end{proof}

For many cases, the result that every class of $\cC(A)$ contains infinitely many height-$1$ prime ideals (but not the exact cardinality) can also be deduced from a more general theorem by Kainrath; see \cite[Theorem 3]{K1}.

\begin{rema} \label{Rem:Iso}
  Suppose that, in addition to the assumptions of the previous theorem,  $A$ is a finitely generated algebra with a known presentation.
  Then there are algorithms to compute the primary decompositions of the ideals $x_1A$, $\ldots\,$,~$x_nA$.
  The associated primes are the height-$1$ prime ideals containing one of the $x_i$.

  Thus, in this case, the previous theorem immediately gives rise to an algorithm for computing the rank of the class group.
  A presentation of the class group, with generators the classes of height-$1$ primes containing one of the $x_i$, can be obtained in the same way, as the entire coefficient matrix $(a_{ij})$ appearing in Theorem~\ref{Thm:CG} can be obtained from the primary decomposition.
\end{rema}

\begin{rema}
  A Krull monoid $H$ is characterized, up to isomorphism, by its unit group $H^\times$, its class group $\cC(H)$, and a family of cardinalities $(m_g)_{g \in \cC(A)}$, where $m_g$ is the cardinality of height-$1$ prime ideals of $H$ in the class $g$.

  Suppose that $A$ is a cluster algebra that is a Krull domain.
  Then our results imply that the multiplicative monoid $A^\bullet$ is determined up to isomorphism by the class group $\cC(A)$, the base ring $K$, and the number of frozen variables $m$.

  In \cite{GLS}, Geiss--Leclerc--Schr\"{o}er show a stronger result for factorial cluster algebras, implying that many factorial cluster algebras are in fact polynomial rings.
\end{rema}

The factorization theory of a Krull domain $A$ is essentially determined by its class group $G=\cC(A)$ and the subset $G_0 = \{\, g \in \cC(A) \ \vert\  \text{$g$ contains a height-$1$ prime ideal} \,\}$.
In particular any question about sets of lengths is completely determined by these two invariants.
We refer to \cite{G} for a recent survey.
If $G_0=G$ is infinite, all arithmetical invariants of $A$ are known to be infinite as well, whereas for factorial domains all these invariants are trivial.
Thus Theorem~\ref{Thm:MainGeneral} exhibits a strong dichotomy between the factorization theory of factorial and non-factorial cluster algebra (as long as the latter is still a Krull domain).

We mention one particular striking result.
Let $a \in A^\bullet$.
We call $k \ge 0$ a \emph{length} of $a$ if there exist atoms $u_1$, $\ldots\,$,~$u_k \in A^\bullet$ such that $a=u_1\cdots u_k$.
The \emph{set of lengths} of $a$, denoted by $\mathsf L(a)$, is defined as the set of all such lengths; we set $\mathsf L(a)=\{0\}$ for $a \in A^\times$.
Then $\mathsf L(a) = \{0\}$ if and only if $a$ is a unit, and $\mathsf L(a) = \{1\}$ if and only if $a$ is an atom.
In any other case $0$,~$1 \not \in \mathsf L(a)$.
Recall that in a Krull domain $\mathsf L(a)$ is always a finite set.
Clearly, if $A$ is factorial, then $\mathsf L(a)$ is a singleton for each $a \in A^\bullet$.
By contrast, a result of Kainrath from \cite{K2}, together with Theorem~\ref{Thm:MainGeneral}, immediately implies the following.

\begin{coro} \label{coro:lengths}
  Let $A$ be a cluster algebra that is a Krull domain and suppose that $A$ is not factorial.
  Let $L \subseteq \Z_{\ge 2}$ be a finite set.
  Then there exists an element $a \in A^\bullet$ such that $\mathsf L(a) = L$.
\end{coro}

A further refinement of this result, that also deals with the multiplicities of the lengths appearing, can be found in \cite[Theorem 7.4.1]{GHK}.

\subsection{\texorpdfstring{$F$}{F}-polynomials}\label{subsection F-poly}
Suppose that the seed $\Sigma=(\x,\y,B)$ has principal coefficients, and let $A=A(\Sigma)$.
For the remainder of this section, we will suppose that $A(\Sigma)$ \emph{is factorial}.
We will see later that this assumption is always fulfilled when $B$ is acyclic, see Corollary~\ref{coro:principal-factorial}.

Cluster algebras were widely studied given their interplay with quiver representations and $2$-Calabi--Yau categories. The connection is given by the Caldero--Chapoton map, that sends objects in the $2$-CY category to elements in the cluster algebra. In that setting, it is important to understand the behavior of $F$-polynomials.

\begin{defi} [$F$-polynomial]
  Let $A$ be a cluster algebra with principal coefficients, and $z \in A$ a non-initial cluster variable.
  The \emph{$F$-polynomial} associated to $z$, that we denote by $F_z$, is obtained from $z \in K[\x^{\pm 1},\y]$ by substituting $x_1=\ldots=x_n=1$.
\end{defi}

\begin{prop}
  Let $z\in A$ be a non-initial cluster variable.
  According to the Laurent phenomenon we may write $z=r/s$ with coprime $r$,~$s\in K[\x,\y]$ such that $s$ is a monomial.
  Then, the element $r\in K[\x,\y]$ is a prime element and not associated to any $x_i$ for $i \in [1,n+m]$.
\end{prop}

\begin{proof}
  Recall first that, by \cite[Corollary 2.3 (ii)]{GLS} two different cluster variables $z_1$,~$z_2\in A$ are not associates.
  This implies in particular that $r$ (and $z$) are not associated to any of $x_1$, $\ldots\,$,~$x_{n+m}$.
  We have that $z$ is a prime element in $A$ by Corollary \ref{Cor:ClusterFactorial}. By Proposition \ref{p-locprime}, the element $z$ remains prime in the localization $K[\x^{\pm 1},\y^{\pm 1}]$ hence it is prime in $K[\x^{\pm 1},\y]$. As $s$ is a unit in $K[\x^{\pm 1},\y]$, the element $r$ is prime in $K[\x^{\pm 1},\y]$.
  We conclude that $r$ is prime in $K[\x,\y]$ unless it is divisible by some $x_i$ with $i\in[1,n]$.
  
  Let $i \in [1,n]$.
  If $x_i$ divides $s$, then it does not divide $r$ in $K[\x,\y]$ by coprimality.
  Suppose now that $x_i$ does not divide $s$, but that it does divide $r$ in $K[\x,\y]$.
  Then $x_i$ divides $r=zs$ also in $A$.
  Since $s$ is a monomial in $x_1$, $\ldots\,$,~$x_n$, and the initial cluster variables are prime elements in $A$, we conclude that $x_i$ also does not divide $s$ in $A$.
  Since $x_i$ is prime in $A$, then $x_i$ divides $z$ in $A$, and hence $x_i$ and $z$ are associates.
  However, then $x_i=z$, contradicting that $z$ is a non-initial cluster variable.
  We conclude that $x_i$ does not divide $r$.
\end{proof}

\begin{theo} Let $A=A(\Sigma)$ be a factorial cluster algebra with principal coefficients. Let $z$ be a non-initial cluster variable and let $F_z$ be the associated $F$-polynomial. Then $F_z$ is a prime element of $K[\y]$.
\end{theo}

\begin{proof}
  A result of Fomin--Zelevinsky, see \cite[Proposition 5.2]{FZ4}, asserts that $F_z \in K[\y]$ is not divisible by a monomial.
  Hence $F_z$ is a prime element in $K[\y]$ if and only if it is prime in $K[\y^{\pm 1}]$.
  To show the latter, we show that $F_z$ is a prime element in $K[\x^{\pm 1}, \y^{\pm 1}]$.

  We put $\widehat{y}_j=x_{j+n}\prod_{i=1}^n x_i^{b_{ij}}$ for $j\in[1,n]$.
  The algebra endomorphism $\alpha$ of $K[\x^{\pm 1}, \y^{\pm 1}]$ satisfying
  \[
    \alpha(x_i) = x_i \text{ for $i \in [1,n]$} \qquad\text{and}\qquad \alpha(x_{j+n}) = \widehat y_j \text{ for $j \in [1,n]$}
  \]
  is an automorphism with $\alpha^{-1}(x_{j+n}) = x_{j+n} \prod_{i=1}^n x_i^{-b_{ij}}$ for $j \in [1,n]$.
  Hence $F_z$ is prime if and only if $\alpha(F_z) = F(\widehat y_1, \ldots, \widehat y_n)$ is prime in $K[\x^{\pm 1}, \y^{\pm 1}]$.
  By \cite[Corollary 6.3]{FZ4}, there is a Laurent monomial $M\in K[\x^{\pm 1},\y^{\pm 1}]^{\times}$ such that $zM$ is equal to $F_z(\widehat{y}_1,\ldots,\widehat{y}_n)$. 
 As noted before, $z$ is prime in $A$ and it is not associated to any of $x_1$, $\ldots\,$, $x_{n+m}$. By Proposition \ref{p-locprime} the element $z$ remains prime in the localization $K[\x^{\pm 1},\y^{\pm 1}]$, and hence $zM$ is prime in $K[\x^{\pm 1},\y^{\pm 1}]$ as well.
\end{proof}

\section{Prime ideals in acyclic cluster algebras}
\label{sec:acyclic-prime}

Throughout this section let $\Sigma = (\x,\y,B)$ be an \emph{acyclic} seed with $\x = (x_1,\ldots, x_n)$ and $\y = (x_{n+1}, \ldots, x_{n+m} )$, and let $A = A(\Sigma)$ be the associated cluster algebra.
For every $x_i$ with $i \in [1,n]$, we denote by $x_i'$ the cluster variable obtained by mutation of $\Sigma$ in direction $i$. 

Recall that the assumption of acyclicity implies that $A$ is a noetherian Krull domain.
In this section we carry out the strategy suggested by Theorem~\ref{Thm:MainGeneral}, and explicitly determine the height-$1$ prime ideals containing one of the exchangeable variables $x_1$, $\ldots\,$,~$x_n$.
The notions of partners and $p$-partners, introduced in Section~\ref{subsection partners}, play a crucial role in this.
In our main result of this section, Theorem~\ref{Thm:Decompositions}, we determine the factorization of $x_iA$ as a divisorial product of height-$1$ prime ideals.

\begin{lemm}[Muller]
\label{Lemma:Muller}
If two $i$,~$j\in [1,n]$ are neighbors, then the corresponding cluster variables $x_i$  and $x_j$ do not lie in a common prime ideal $\prim \subseteq A$.
\end{lemm}

\begin{proof}
  The lemma and the proof we present are due to Muller \cite[Lemma 5.3]{M}.
  We reproduce the proof for the convenience of the reader. Suppose that there is a prime ideal $\prim\subseteq A$ that contains two exchangeable variables $x_i$ and $x_j$ such that there is an arrow $i\to j$ in $\Gamma(B)$.
  Let $i_1=i$ and $i_2=j$, and consider paths $i_1\to i_2\to\ldots\to i_r$ in $\Gamma(B)$ such that $i_s$ is exchangeable and $x_{i_s}\in\prim$ for all $s\in[1,r]$.
  Since $\Sigma$ is acyclic, there exists such a path of maximal length $r$.
  By assumption $r\geq 2$.
  We have
\[
\prod_{t\in N_{+}(i_r)}x_t^{b_{ti_r}}=x_{i_r}x_{i_r}'-\prod_{t\in N_{-}(i_r)}x_t^{-b_{ti_r}}\in \prim
\]
because $x_{i_{r-1}}$,~$x_{i_r}\in\prim$ and $i_{r-1}\in N_{-}(i_r)$. Since $\prim$ does not contain $1$, we can deduce that $N_{+}(i_r)$ is non-empty. Since $\prim$ is prime, there is a $t\in N_{+}(i_r)$ with $x_t\in\prim$. This $t$ cannot be frozen, because otherwise $x_t x_t^{-1}=1\in\prim$. Hence it is exchangeable, which contradicts the maximality of the chosen path.
\end{proof}

Nevertheless cluster variables corresponding to non-neighboring vertices can lie in a common prime ideal, as the next example shows.

\begin{exam}
  Let $Q=(1\rightarrow 2\leftarrow 3)$ be of type $A_3$ and let $\x=(x_1,x_2,x_3)$.
  According to Theorem~\ref{Thm:BFZ} the cluster algebra $A$ of $Q$ is generated by $x_i$,~$x_i'$ with $i\in[1,3]$ subject to relations $x_1x_1'=x_2+1=x_3x_3'$ and $x_2x_2'=1+x_1x_3$. Then the ideal $\prim=\langle x_1,x_3\rangle \subseteq A$ is prime because $A/\prim\cong K[x_1',x_2,x_2',x_3']/\langle x_2+1,x_2x_2' -1\rangle \cong K[x_1',x_3']$ is a domain.
\end{exam}

We shall be concerned with the following ideals.
Note that if $p$,~$q \in K[\x,\y]$ are prime elements that are not associated to any of $x_1$, $\ldots\,$,~$x_{n+m}$, then $p$ and $q$ are associates in $K[\x,\y]$ if and only if they are associates in $K[\x^{\pm 1},\y^{\pm 1}]$ if and only if they are associates in $A$.
Thus,  $\ptns{p}=\ptns{q}$ for the set of $p$-partners if $pA=qA$.
Hence the notation $\prim_{pA}(I)$ in the following definition makes sense.

\begin{defi}[Prime ideals]
  \mbox{}
  \begin{enumerate}
    \item For every partner set $\ps \subseteq [1,n]$ and every subset $I\subseteq \ps$, let
      \begin{align*}
        A_{\overline{I}}=A\big[(x_i')^{-1},x_j^{-1} \ \vert \ i\in I,\, j\in [1,n]\setminus I\big].
      \end{align*}
    \item 
      For every prime element $p \in K[\x,\y]$ that is not associated to any of $x_1$, $\ldots\,$,~$x_{n+m}$, and every subset $I \subseteq \ptns{p}$ of the set of $p$-partners, define
      \[
        \PP{p}{I} = \PP{pK[\x,\y]}{I} = \PP{pA}{I} = A \cap p A_{\overline I}.
      \]
    \end{enumerate}
\end{defi}

By Lemma \ref{Lemma:MutatedPartners}, the set $\{\, x_i', x_j \ \vert \  i\in I,\, j\in [1,n+m] \setminus I \,\}$ is a cluster of $A$, so the ring obtained by localizing with respect to the multiplicative set generated by the cluster, that is $A_{\overline I}$, is a Laurent polynomial ring.

Let us summarize some basic properties of the ideals $\PP{p}{I}$.

\begin{prop}
  \label{Prop:PartnerIdeals}
  Let $p$, $q \in K[\x,\y]$ be prime elements not associated to any of $x_1$, $\ldots\,$,~$x_{n+m}$, and let $I \subseteq \ptns{p}$ and $J \subseteq \ptns{q}$.
  \begin{enumerate}
  \item  \label{PartnerIdeals:Prime} The ideal $\PP{p}{I}$ is a height-$1$ prime ideal of $A$.
  \item \label{PartnerIdeals:Cont}The prime ideal $\PP{p}{I}$ contains
    \[
      \{\, p, x_i,  x_j' \mid i \in I, j \in \ptns{p} \setminus I \,\}.
    \]
  \item \label{PartnerIdeals:Equal} We have $\PP{p}{I}=\PP{q}{J}$ if and only if $pA=qA$ and $I=J$.
  \end{enumerate}
\end{prop}

\begin{proof}
\ref*{PartnerIdeals:Prime}
  Since $p$ is a prime element in $K[\x,\y]$ and $p$ is not divided by any of $x_1$, $\ldots\,$,~$x_{n+m}$ by assumption, the element $p$ is also a prime element of $A_{\overline{I}}$.
  By Krull's Principal Ideal Theorem, the prime ideal $p A_{\overline I}$ has height $1$.
  Since $A_{\overline I}$ is a localization of $A$, therefore the same is true for the contraction $\PP{p}{I} = A \cap pA_{\overline I}$.

\ref*{PartnerIdeals:Cont}
  That $p \in \PP{p}{I}$ is clear from the definitions.

  Let $i\in I$.
  Since $p$ divides the exchange polynomial $f_i$, we have
  \[
    x_ix_i' =f_i \in pA \subseteq \PP{p}{I},
  \]
  and hence $x_i \in \PP{p}{I}$ or $x_i' \in \PP{p}{I}$.
  The latter is impossible, since $x_i' \in A_{\overline I}^\times$ while $p \not \in A_{\overline I}^\times$, hence $x_i \in \PP{p}{I}$.
  
  For $j \in \ptns{p}\setminus I$ we argue analogously $x_j' \in \PP{p}{I}$, since $x_j \in A_{\overline I}^\times$.

\ref*{PartnerIdeals:Equal}
  If $pA=qA$ and $I=J$, then trivially $\PP{p}{I}=\PP{q}{J}$.

  Suppose now $\PP{p}{I} = \PP{q}{J}$.
  We have $q \in \PP{p}{I}$ and hence $q \in pA_{\overline I}$.
  Since $p$ and $q$ are prime elements in $A_{\overline I}$, we conclude that they are associated.
  Since neither of them is divisible by a monomial, we have $p = q \lambda$ with $\lambda \in K^\times$.
  In particular, $pA=qA$.

  Assume that $I \neq J$.
  Then there exists an element $j\in J\setminus I$ or an element $i\in I\setminus J$.
  By symmetry, we may without loss of generality assume that there is an element $j\in J\setminus I$.
  Then $x_j' \in \PP{p}{I}$ and $x_j' \in A_{\overline J}^\times$.
  From this we can deduce that $\PP{p}{I} A_{\overline{J}}=A_{\overline{J}}$.
  On the other hand we have $\PP{p}{J} A_{\overline{J}}=p A_{\overline{J}}\neq A_{\overline{J}}$ because $p$ is not a unit in $A_{\overline{J}}$.
  It follows that $\PP{p}{I} \neq \PP{q}{J}$.
  \qedhere
\end{proof}

The brunt of the work now lies in showing that the prime ideals just introduced are indeed precisely the height-$1$ prime ideals containing some exchangeable variable $x_1$, $\ldots\,$,~$x_n$.
\begin{lemm}
  \label{Lemma:PrimesAndClusterVars}
  Let $i \in [1,n]$.
  Every height-$1$ prime ideal of $A$ containing $x_i$ is equal to some $\PP{p}{I}$, where $p \in K[\x,\y]$ is a prime element dividing the exchange polynomial $f_i$, and $I \subseteq \ptns{p}$ is a subset with $i \in I$.
\end{lemm}

\begin{proof}
  We prove the claim by induction on the number of pairwise non-associated prime elements of $K[\x,\y]$ that divide one of the exchange polynomials $f_1$, $\ldots\,$,~$f_n$.
  That is, by induction on
  \[
    c = \card{  \{\, \mathfrak P \in \Xx(K[\x,\y]) \ \vert\ f_i \in \mathfrak P \text{ for some $i \in [1,n]$} \,\} }.
  \]
  If $c=0$, then $n=0$ and there is nothing to show.
  Suppose from now on $c \ge 1$ and that the claim has been established for $c-1$.

  Without restriction we show the claim for $x_1$.
  By \ref{PartnerIdeals:Cont} of Proposition~\ref{Prop:PartnerIdeals}, each of the stated ideals $\PP{p}{I}$ contains $x_1$.
  It therefore suffices to show that $x_1$ is not contained in any other height-$1$ prime ideal.

  If $i$ is a neighbor of $1$, then $x_1$ and $x_i$ are not contained in a common prime ideal of $A$ by Lemma~\ref{Lemma:Muller}.
  Thus, if $S \subseteq A^\bullet$ denotes the multiplicative subset generated by $\{\, x_i \ \vert\  i \in N(1) \,\}$, it suffices to show that $x_1$ is not contained in any height-$1$ prime ideal of $S^{-1}A$, other than an ideal of the form $S^{-1} \PP{p}{I}$ with $p$ and $I$ as above.
  Since $A$ is acyclic, Theorem~\ref{Thm:BFZ} implies that $A$ is generated by $\{\, x_i, x_i', x_j^{\pm 1}\ \vert\  i \in [1,n], j \in [n+1,n+m] \,\}$ as a $K$-algebra.
Therefore, the algebra $S^{-1}A$ is isomorphic to the cluster algebra obtained by freezing the variables $x_i$ for $i \in N(1) \cap [1,n]$.
  Moreover $S^{-1}\PP{p}{I} = S^{-1}A \cap pA_{\overline I}$ by Proposition~\ref{p-locprime}.
  Replacing $A$ by $S^{-1}A$, we may without restriction assume that all neighbors of $1$ are frozen.
  This implies in particular that $f_i \in K[\y]$ for all partners $i$ of $1$.
  Moreover, if $i$ is a partner of $1$, then $x_i$ does not occur in any exchange polynomial of $\Sigma$.

  Let now $\prim \subseteq A$ be a height-$1$ prime ideal containing $x_1$.
  Then $\prim$ also contains the exchange polynomial $f_1=x_1x_1'$, and hence $\prim$ contains a prime element $p \in K[\y]$ dividing $f_1$.
  For each $i \in \ptns{p}$, we have that $p$ divides $x_ix_i'$ and hence $x_i \in \prim$ or $x_i' \in \prim$.
  Let $I = \{\, i \in \ptns{p} \ \vert\ x_i \in \prim \,\}$ and note $1 \in I$.
  Let $\mathfrak a = \langle x_i, x_j', p \ \vert\ i \in I, j \in \ptns{p} \setminus I \rangle_A$.
  Then $\mathfrak a \subseteq \prim$ by construction, and  $\mathfrak a \subseteq \PP{p}{I}$ by \ref{PartnerIdeals:Cont} of Proposition~\ref{Prop:PartnerIdeals}.
  We will show $\mathfrak a = \PP{p}{I}$.
  Then $\PP{p}{I} \subseteq \prim$, and, since $\PP{p}{I}$ and  $\mathfrak p$ are both height-$1$ prime ideals, we can conclude $\prim = \PP{p}{I}$.

  We may without restriction assume $\ptns{p} = [1,r]$ for $r \ge 1$.
  Consider the exchange matrix $B'$ obtained from $B$ by erasing the first $r$ columns and $r$ rows.
  Note that the first $r$ rows of $B$ are in fact all $0$.
  Let $\x'=(x_{r+1},\ldots, x_n)$, let $\Sigma' = (\x', \y, B')$, and let $A'$ be the cluster algebra associated to $\Sigma'$.
  Again using that $A$ is generated by $\{\, x_i, x_i', x_j^{\pm 1}\ \vert\  i \in [1,n], j \in [n+1,n+m] \,\}$ as a $K$-algebra, as well as the analogous statement for $A'$, we see that $A'$ is a subalgebra of $A$ and
  \[
    A = A'\big[ x_i, x_i' \ \vert \ i \in [1,r] \big].
  \]
  Using the presentation of $A$ given by Berenstein--Fomin--Zelevinsky in Theorem~\ref{Thm:BFZ},
  \begin{equation} \label{Eq:Aext}
    A = A'\big[ x_i, x_i' \ \vert \ i \in [1,r] \big] \cong A'\big[X_i, X_i' \ \vert\ i \in [1,r] \big] / \langle X_i X_i' - f_i \ \vert\ i \in [1,r] \rangle,
  \end{equation}
  for indeterminates $X_1$, $\ldots\,$, $X_r$, $X_1'$, $\ldots\,$,~$X_r'$ that are algebraically independent over $A'$.
  
  As we have removed all $p$-partners, the element $p$ is not a divisor of an exchange polynomial of $A'$, and the induction hypothesis therefore applies to $A'$.

  \begin{claim}
    $p$ is a prime element in $A'$.
  \end{claim}

  \begin{proof}
    The prime element $p \in K[\y]$ is not a monomial, hence $p$ is a prime element in the Laurent polynomial ring $K[x_{r+1}^{\pm 1},\ldots, x_{n+m}^{\pm 1}] = A'[x_{r+1}^{-1}, \ldots, x_n^{-1}]$.
    Thus
    \[
      \qrim = pA'\big[x_{r+1}^{-1}, \ldots, x_n^{-1}]
    \]
    is a height-$1$ prime ideal in $A'[x_{r+1}^{-1}, \ldots, x_n^{-1}]$.
    Since $A'$ is a Krull domain, the ideal $pA'$ can be written as a divisorial product of height-$1$ primes of $A'$.
    Proposition~\ref{p-locprime} implies
    \[
      pA' = (\qrim \cap A') \cdot_v \vprod_{s=1}^t \qrim_s
    \]
    for some height-$1$ prime ideals $\qrim_s \subseteq A'$ (with $t \ge 0$ and $s \in [1,t]$) such that  $\qrim_s$ contains one of $x_{r+1}$, $\ldots\,$,~$x_n$.
    To show that $p$ is a prime element in $A'$, it will suffice to show that $t = 0$.

    Applying the induction hypothesis to $A'$, we see that any prime ideal $\qrim_s$ containing one of $x_{r+1}$, $\ldots\,$,~$x_n$ is of the form
    \[
      \qrim_s = A'  \cap q A'_{\overline J}
    \]
    for a prime element $q \in K[\x',\y]$ dividing one of the exchange polynomials of $\Sigma'$, and $\emptyset \ne J \subseteq \ptns{q}$, where the $q$-partners $\ptns{q}$ are considered in $A'$, that is $\ptns{q} \subseteq [r+1,n]$.
    Here,
    \[
      A'_{\overline J}  =A'\big[(x_j')^{-1},x_i^{-1}\ \vert \ j\in J,\, i \in [r+1,n] \setminus J \big].
    \]

    To show $p \not \in \fq_s$, it suffices to show $p \not \in q A'_{\overline J}$.
    Suppose to the contrary that $p \in qA'_{\overline J}$.
    Since $p \in K[\y]$ is a prime element in $A'_{\overline J}$, it follows that $p=\lambda q$ for some $\lambda \in (A'_{\overline J})^\times$.
    But the group of units is the multiplicative abelian group generated as
    \[
      (A'_{\overline J})^{\times}=\left\langle\mu,\, x_j',\,  x_i \ \vert \ \mu \in K^\times,\, j\in J,\, i\in [r+1,n+m]\setminus  J \right\rangle_{\Z}. 
    \]
    It follows that $\lambda \in K^\times$ and thus $p$ and $q$ are associated in $K[\x,\y]$.
    This is impossible, in contradiction to the assumption $p \in qA'_{\overline J}$.
    \renewcommand{\qedsymbol}{\ensuremath{\square} (Claim)}
  \end{proof}

  From Equation~\eqref{Eq:Aext} we see that
  \[
    A / \mathfrak a \cong (A'/pA')[ x_i', x_j \mid i \in I,\,  j \in \ptns{p} \setminus I ]
  \]
  is isomorphic to a polynomial ring over the domain $A'/pA'$.
  Thus, the ideal $\mathfrak a$ is a prime ideal in $A$.
  Since $\PP{p}{I}$ is a height-$1$ prime ideal containing $\mathfrak a$, we must have $\mathfrak a = \PP{p}{I}$.
\end{proof}

\begin{coro}
  If $i$,~$j \in [1,n]$ are not partners, then there does not exist a height-$1$ prime ideal $\prim\subseteq A$ such that $\{x_i,x_j\}\subseteq \prim$.
\end{coro}

\begin{proof}
  This follows from Lemma~\ref{Lemma:PrimesAndClusterVars} together with \ref{PartnerIdeals:Equal} of Proposition~\ref{Prop:PartnerIdeals}.
\end{proof}

In the proof of Lemma~\ref{Lemma:PrimesAndClusterVars}, if all exchange polynomials $f_i$ are contained in $K[\y]$, and thus in particular if there is only one partner set, the localization step becomes unnecessary.
In this case, we see $\PP{p}{I} =\langle x_i, x_j', p \ \vert \ i\in I,\, j\in \ptns{p}\backslash I\rangle_A$.
In general, when $\Sigma$ admits more than one partner set, we can show that $\PP{p}{I}$ is generated by the same set as a divisorial ideal.

\begin{prop} \label{Prop:PrimeGens}
  Let $p \in K[\x,\y]$ be a prime element dividing one of the exchange polynomials of the seed $\Sigma$.
  Let $\emptyset \ne I \subseteq \ptns{p}$.
  Then
  \begin{align*}
    \PP{p}{I}=\big(\langle x_i, x_j', p \ \vert \ i\in I,\, j\in \ptns{p}\setminus I \rangle_A\big)_v.
  \end{align*}
\end{prop}

\begin{proof}
  Let $\mathfrak a = \big(\langle x_i, x_j', p \ \vert \ i\in I,\, j\in \ptns{p} \setminus I \rangle_A\big)_v$.
  Since $\mathfrak a$ is a divisorial ideal and $A$ is a Krull domain, the ideal $\mathfrak a$ is a divisorial product of height-1 prime ideals.
  Our assumption that $I \ne \emptyset$ implies that there exists an $i \in \ptns{p}$ such that $x_i \in \mathfrak a$.
  Denoting as always by $f_i$ the exchange polynomial of $x_i$, let $f_i=p_1 \cdots p_r$ with prime elements $p=p_1$, $p_2$, $\ldots\,$,~$p_r \in K[\x,\y]$.
  By \ref{prop:exchpoly:rep} of Proposition~\ref{prop:exchpoly}, these prime elements are pairwise non-associated.
  Lemma~\ref{Lemma:PrimesAndClusterVars} implies
  \[
    \mathfrak a = \vprod_{s=1}^r \vprod_{\emptyset \ne J \subseteq \ptns{p_s}} \PP{p_s}{J}^{k_{s,J}}\quad\text{with $k_{s,J} \ge 0$}.
  \]
  Localizing, we have
  \[
    \mathfrak a A_{\overline J} = \vprod_{\substack{s=1\\J \subseteq \ptns{p_s}}}^r     (\PP{p_s}{J} A_{\overline J})^{k_{s,J}}.
  \]
  for all $j \in [1,r]$ and $\emptyset \ne J \subseteq \ptns{p_j}$.

  On the other hand, from the definition of $\mathfrak a$ and $A_{\overline I}$, we see $\mathfrak a A_{\overline I} = \PP{p}{I} A_{\overline I}$ and $\mathfrak a A_{\overline J} = A_{\overline J}$ for $\emptyset \ne J \subseteq \ptns{p}$ with $I \ne J$.
  If $p_j$ is not associated to $p$ and $\emptyset \ne J \subseteq \ptns{p_j}$, then $p$ is a prime element in $A_{\overline J}$ with $p_jA_{\overline J} \ne pA_{\overline J}$.
  Hence $pA_{\overline J} \not\subseteq p_j A_{\overline J}$.
  Comparing coefficients, we find $k_{1,I}=1$ and $k_{j,J}=0$ for all other coefficients.
  Thus $\mathfrak a = \PP{p}{I}$.
\end{proof}

\begin{ques}
We know that $\{\, x_i, x_j', p \ \vert\ i \in I, j \in \ptns{p}\setminus I \,\}$ generates $\PP{p}{I}$ as a divisorial ideal, and that, if all exchange polynomials are contained in $K[\y]$, the same set generates $\PP{p}{I}$ as an ideal.
  Is it always the case that $\PP{p}{I}$ is generated by this set as an ideal?
\end{ques}

\begin{theo}
  \label{Thm:Decompositions}
    Let $\Sigma=(\x,\y,B)$ be an acyclic seed with $\x=(x_1,\ldots,x_n)$ and $\y=(x_{n+1},\ldots, x_{n+m})$, and let $A =A(\Sigma)$.
For every exchangeable index $i\in [1,n]$ the following equation holds.
\[
  x_iA=\vprod_{\substack{\mathfrak P \in \Xx(K[\x,\y]) \\ f_i \in \mathfrak P}} \ \, \vprod_{\substack{I\subseteq \ptns{\mathfrak P}\\ i\in I}}\PP{\mathfrak P}{I}
\]
\end{theo}

\begin{proof}
  We proceed in the same way as in the proof of Proposition~\ref{Prop:PrimeGens}.
  By \ref{prop:exchpoly:rep} of Proposition~\ref{prop:exchpoly}, we can write $f_i = p_1 \cdots p_r$ with pairwise non-associated prime elements $p_j \in K[\x,\y]$.

  By Lemma~\ref{Lemma:PrimesAndClusterVars}, we have
  \[
    x_iA=\vprod_{s=1}^r \vprod_{\substack{I \subseteq \ptns{p_s}\\ i\in I}}\PP{p_s}{I}^{k_{s,I}} \qquad \text{with $k_{s, I} \ge 0$.}
  \]
  Localizing, for $j \in [1,r]$ and $I \subseteq \ptns{p_j}$ with $i \in I$,
  \[
    x_i A_{\overline I} = f_i A_{\overline I} = \prod_{s=1}^r p_s A_{\overline I}.
  \]
  Since $\PP{p_j}{I} A_{\overline I}= p_j A_{\overline I}$, a comparison of coefficients gives $k_{j,I} = 1$.
\end{proof}

\section{Class groups of acyclic cluster algebras}
\label{sec:class-group-acyclic}

We keep the notation of the previous section: let $\Sigma = (\x,\y,B)$ be an \emph{acyclic} seed with $\x = (x_1,\ldots, x_n)$ and $\y = (x_{n+1}, \ldots, x_{n+m} )$, and let $A = A(\Sigma)$ be the associated cluster algebra.

In this section we combine Theorem~\ref{Thm:MainGeneral} with the description of height-$1$ prime ideals over the exchangeable variables $x_1$, $\ldots\,$,~$x_n$, obtained in the previous section, to determine the rank of the class group in the acyclic case in terms of the initial exchange polynomials, and hence in terms of the initial exchange matrix $B$ together with knowledge about roots of unity in the base ring $K$.
When $B$ is skew-symmetric, and thus arising from an ice quiver $Q$, the resulting description is in terms of quiver combinatorics.
This yields a particularly explicit description in the cases where $K$ is one of $\Z$, $\Q$, or an algebraically closed field.

We then give some easier to state special cases of this theorem, and apply it to several examples.
In particular, we recover all known results on class groups of cluster algebras and the classification of factoriality for cluster algebras of (extended) Dynkin type.

Geiss--Leclerc--Schr\"oer in \cite{GLS} give necessary conditions for a cluster algebra to be factorial.
For an acyclic seed, we can now show that these conditions are also sufficient.
(For seeds where $\Gamma(B)$ contains an oriented cycle the conditions are not sufficient, as the example of a cluster algebra of type $A_3$ represented by a cycle easily shows.)

\begin{theo} \label{Thm:AcyclicFactorial}
  Let $A$ be a cluster algebra with acyclic seed $\Sigma$.
  Then the following statements are equivalent.
  \begin{equivenumerate}
  \item $A$ is factorial.
  \item The exchange polynomials $f_1$, $\ldots\,$,~$f_n$ are prime elements in $K[\x,\y]$ and pairwise distinct.
  \end{equivenumerate}
\end{theo}

\begin{proof}
  By Theorem~\ref{Thm:MainGeneral} the algebra $A$ is factorial if and only if each of $x_1$, $\ldots\,$,~$x_n$ is contained in a unique height-$1$ prime ideal. The claim therefore follows from Lemma~\ref{Lemma:PrimesAndClusterVars}.
\end{proof}

\begin{coro}
  Let $Q$ be an acyclic ice quiver without parallel arrows.
  Then the cluster algebra $A=A(Q)$ is factorial  if and only if every partner set $\ps$ in $Q$ is a singleton.
  In other words, $Q$ admits no partners $i\neq j$.
\end{coro}

\begin{proof}
  In this case all exchange polynomials are prime elements by Corollary~\ref{coro:parallel-irred}.
  The assumption that all partner sets are singletons implies that the exchange polynomials are pairwise distinct and conversely.
\end{proof}

\begin{coro} 
\label{coro:principal-factorial}
Suppose that the seed $\Sigma=(\x,\y,B)$ is acyclic and has principal coefficients. Then the cluster algebra $A=A(\Sigma)$ is factorial.
\end{coro}
\begin{proof}
In this case the exchange polynomials are pairwise distinct prime elements in $K[\x,\y]$ by Corollary~\ref{coro:principal_irred}. 
\end{proof}

More generally, we are able to determine the rank of the class group.
Recall that $\mu_d^*(K)$ denotes the set of $d$-th primitive roots of unity in $K$.
\begin{defi}[$\nu_K(d)$] 

  Let $\nu_K(d)$ denote the number of irreducible factors of the cyclotomic polynomial $\Phi_d$ over $K$.
  \end{defi}
  
  In other words, the number $\nu_K(d)$ is
  \[
    \nu_K(d)  =
    \begin{cases}
      1 & \text{ if $\mu_d^*(K) = \emptyset$,}\\
      \varphi(d) = \card{(\Z/d\Z)^\times} & \text{ if $\mu_d^*(K) \ne \emptyset$.}
    \end{cases}
  \]

  Recall also that $\bgcd_i = \gcd( b_{ji} \mid j \in [1,n+m] )$ denotes the greatest common divisor of the $i$-th column of the exchange matrix $B$ for $i \in [1,n]$ (see Definition~\ref{defi:bgcd}).
  Partner sets were introduced in Definition~\ref{defi:partners}.
  As always, our standing assumption that $[1,n]$ contains no isolated indices if $K \ne \Z$ remains in effect.

\begin{theo} \label{Thm:MainAcyclic}
  Let $\Sigma=(\x,\y,B)$ be an acyclic seed with $\x=(x_1,\ldots,x_n)$ and $\y=(x_{n+1},\ldots, x_{n+m})$, and let $A =A(\Sigma)$.
  For a partner set $\ps \subseteq [1,n]$ and $d \in \Z_{\ge 1}$, let
  \begin{itemize}
  \item $\dcount(\ps,d)$ denote the number of $i \in V$ for which $d$ divides $\bgcd_i$,
    \item $\vexp(V) = \val_2(\bgcd_i)$ be the $2$-valuation of $\bgcd_i$ for $i \in V$ \textup{(}this is independent of $i$\textup{)}.
  \end{itemize}
  Then the class group of $A$ is a finitely generated free abelian group of rank
  \[
    r= \sum_{\substack{V \\ \text{$V$ a partner set}}} r_V,
  \]
 where
 \[
    r_V = 2^{\card{V}} -  1 - \card{V} \quad\text{if $V$ is the partner set of isolated indices},
  \]
  and otherwise
  \[
    r_V =   \sum_{\substack{d \in \Z_{\ge 1} \\ d \text{ odd}}} \big(2^{\dcount(\ps,d)} - 1) \, \nu_K(2^{\vexp(\ps)+1} d)  - \card{V}.
  \]
\end{theo}

\begin{proof}
  By Theorem~\ref{Thm:MainGeneral} the class group of $A$ is a finitely generated free abelian group, and its rank is $r=t -n$, where $t$ is the number of height-$1$ prime ideals containing one of the exchangeable variables $x_1$, $\ldots\,$,~$x_n$.
  Stated another way, $r=\sum_{V} r_V$ with the sum running over all partner sets $V$, and
  \[
    r_V = \card{\{ \prim \in \Xx(A) \ \vert\ x_i \in \prim \text{ for some $i \in V$}  \}} - \card{V}.
  \]
  We proceed to determine $r_V$.

  First assume $K=\Z$ and that $V$ is the partner set of isolated indices.
  Then $f_i=2$ is a prime element of $\Z[\x,\y]$ for every $i\in V$.
  Thus $V= \ptns{2}$ is the set of $2$-partners, and Lemma~\ref{Lemma:PrimesAndClusterVars} implies that there are $2^{\card{V}} - 1$ prime ideals containing some $x_i$ with $i \in V$.
  Thus
  \[
    r_V = 2^{\card{V}} - 1 - \card{V}.
  \]

  Let now $K$ be $\Z$ or a field of characteristic $0$, and let $V$ be a partner set distinct from the one containing the isolated vertices.
  By Proposition~\ref{prop:exchpoly} there exist monomials $g$, $h \in K[\x,\y]$ with disjoint support, at least one of which is non-trivial, such that
  \[
    f_i = g^{\bgcd_i} + h^{\bgcd_{i}} \quad\text{for every $i \in V$.}
  \]
  Moreover, $\val_2(\bgcd_i) = \val_2(\bgcd_j) = \vexp(V)$ for all $i$,~$j \in V$.
  For $i \in V$ let $\bgcd_i = 2^{\vexp(V)} c_i$ with $c_i$ odd.
  By Lemma~\ref{lemma:poly}, over $\Z$, the exchange polynomial factors as
  \[
    f_i = \prod_{d \mid c_i} \Phi_{2^{\vexp(V)+1}d}(g,h).
  \]
  By the same lemma, the polynomial $\Phi_{2^{\vexp(V)+1}d}(g,h)$ factors into a product of $\nu_K(2^{\vexp(V)+1}d)$ irreducible polynomials over $K$.

  Since $\Phi_{2^{\vexp(V)+1}d}(g,h)$ for an odd integer $d \ge 1$ appears as a factor of $f_i$ if and only if $d \mid \bgcd_i$, we conclude that it appears as a factor in $\dcount(V,d)$ of the exchange polynomials.
  By Lemma~\ref{Lemma:PrimesAndClusterVars}, we conclude that there are $(2^{\dcount(V,d)} - 1)\nu_K(2^{\vexp(V)+1}d)$ height-$1$ prime ideals containing $\Phi_{2^{\vexp(V)+1}d}(g,h)$ and one of the exchangeable variables $x_i$ with $i$ a $\Phi_{2^{\vexp(V)+1}d}(g,h)$-partner.
  Summing over all possible odd integers $d \ge 1$, we find the desired expression for $r_V$.
\end{proof}

\begin{rema}
  If $K=\Z$ or $K=\Q$, then $\nu_K(2^{\vexp(V)+1}d) = 1$ for all $V$ and $d$.
  If $K$ is algebraically closed, then $\nu_K(2^{\vexp(V)+1}d) = \varphi(2^{\vexp(V)+1}d) = 2^{\vexp(V)} \varphi(d)$ for all $V$ and odd $d \in \Z_{\ge 1}$.
\end{rema}

The expressions for $r_V$ reduce to simpler ones in two important special cases.
Recall that a non-constant exchange polynomial $f_i$ is irreducible if and only if $x^{\bgcd_i} + 1$ is irreducible over $K$.
In the cases that  $K$ is one of $\Z$, $\Q$, or an algebraically closed field, this can easily be expressed purely in terms of $\bgcd_i$; see Proposition~\ref{prop:exchpoly}.

\begin{coro}
  \label{teorema:rank}
  If all exchange polynomials corresponding to indices in the partner set $V$ are prime elements of $K[\x,\y]$, then
  \[
    r_V= 2^{\card{\ps}}-1-\card{\ps}.
  \]
\end{coro}

\begin{proof}
  In the case that $V$ is the partner set of isolated indices, there is nothing to show.

  Suppose therefore that $V$ is a partner set containing non-isolated vertices.
  By definition of a partner set and the irreducibility of the exchange polynomials, we must have $f_i=f_j$ for all $i$,~$j \in V$.
  By Lemma~\ref{Lemma:PrimesAndClusterVars}, there are $2^{\card{V}} - 1$ height-$1$ prime ideals in $A$ that contain some $x_i$ for $i \in V$.
  On the other hand, there are $V$ such cluster variables, each of them reducing the rank by $1$.
  Thus, $r_V = 2^{\card{V}}  - 1 - \card{V}$.

  Alternatively, one may observe that the exchange polynomials have to be irreducible over $\Z[\x,\y]$, that therefore $\bgcd_i$ is a $2$-power,  and hence $c(V,1) = \card{V}$ and $c(V,d) = 0$ for any odd $d > 1$.
  Thus, in the expression for $r_V$, only a single term is left over, and since the exchange polynomial is assumed to be irreducible, necessarily also $\nu_K(2^{\vexp(V)+1}) = 1$ for this term.
\end{proof}

Recall that $\sigma_0(c)$ denotes the number of positive divisors of $c \in \Z_{>0}$.

\begin{coro}\label{Corol singleton}
  If a partner set $V=\{i\}$ of non-isolated indices is a singleton, let $\bgcd_i = 2^{\vexp(V)} c$ with $c$ odd.
  Then
  \[
    r_V = \sum_{\substack{d \in \Z_{\ge 1}\\ d \mid \bgcd_i, \text{ $d$ odd}}} \nu_K(2^{\vexp(V)+1} d) - 1 =
    \begin{cases}
      \sigma_0(c) - 1 & \text{if $K=\Z$ or $K=\Q$,}\\
      \bgcd_i  - 1& \text{if $K$ is algebraically closed.}
    \end{cases}
  \]
\end{coro}

\begin{proof}
  Since $\dcount(V,d) = 1$ if $d \mid \bgcd_i$ and $\dcount(V,d)=0$ otherwise, the expression for $r_V$ reduces to
  \[
    r_V = \sum_{\substack{d \in \Z_{\ge 1}\\ d \mid \bgcd_i, \text{ $d$ odd}}} \nu_K(2^{\vexp(V)+1} d) - 1.
  \]
  If $K$ is algebraically, closed, then $\nu_K=\varphi$ is multiplicative (on coprime integers), and
  \[
    r_V = 2^{\vexp(V)} \sum_{\substack{d \in \Z_{\ge 1}\\ d \mid \bgcd_i, \text{ $d$ odd}}} \varphi(d) - 1 = \bgcd_i - 1.
  \]

  If $K=\Z$ or $K=\Q$,
  \[
    r_V = \sum_{\substack{d \in \Z_{\ge 1}\\ d \mid \bgcd_i, \text{ $d$ odd}}} 1  - 1 = \sigma_0(c) - 1.
  \]
\end{proof}

\begin{exam}
\label{ex_rank_two}
  Let $n \ge 1$,
  \[
    B = \begin{bmatrix} 0 & n \\ -1 & 0 \end{bmatrix},
  \]
  and let $A=A(B)$ be the associated cluster algebra.
  Then $f_1 = 1+x_2$ and $f_2 = 1 + x_1^n$.
  The partner sets are $V_1=\{1\}$ and $V_2=\{2\}$, with $\bgcd_1 = 1$ and $\bgcd_2 = n$.
  Observe that $r_{V_1} = 0$, so that the rank of the class group is $r = r_{V_2}$.

  If $K$ is an algebraically closed field, then $r_{V_2}= n - 1$, corresponding to the fact that $f_2$ factors as a product of $n$ irreducible polynomials.

  Consider now the case where $K=\Z$ or $K=\Q$.
  Write $n = 2^lc$ with $l \ge 0$ and $c$ odd.
  Then $r_{V_2} = \sigma_0(c) -1$.
  This corresponds to the fact that $f_2$ has $\sigma_0(c)$ irreducible factors.
  In particular, $A$ is factorial if and only if $n$ is a power of $2$, and the rank of $\cC(A)$ is $1$ if and only if only if $n=2^lp$ for $l \ge 0$ and some odd prime $p$.
\end{exam}

\begin{exam}
  Let $Q$ be the $n$-Kronecker quiver  $Q=(1\, \xrightarrow{ \ n \ }\, 2)$, and $A=A(Q)$.
  Then $f_1=1+x_2^n$ and $f_2=1+x_1^n$.
  The rank of $\cC(A)$ is twice that of the previous example, so that the rank of $\cC(A)$ is $2(n-1)$ if $K$ is algebraically closed and $2(\sigma_0(c) - 1)$ if $K =\Z$ or $K=\Q$ (with $c$ again denoting the odd part of $n$).
\end{exam}

\begin{exam} Consider the cluster algebra over $\mathbb{Q}$ associated to the quiver from Example~\ref{exam:bigquiver}.
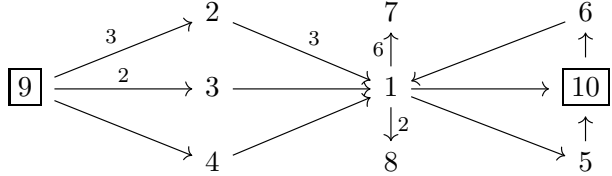
  \begin{center}
    \begin{tikzcd}[row sep = small]
      && 2 \arrow[drr, "3"]  
      & & 7  && 6 \arrow[dll]  \\
      \boxed{9} \arrow[rr, "2"]\arrow[urr, "3"] \arrow[drr] && 3  \arrow[rr]
      & & 1\arrow[rr] \arrow[drr] \arrow[d, "2"] \arrow[u, "6"] & & \boxed{10}\arrow[u]\\
    && 4 \arrow[urr] && 8  && 5 \arrow[u]\end{tikzcd}
  \end{center}
  Recall that we had three partner sets with 2 indices each $V_1 =\{2,4 \}$, $V_2=\{5,6 \}$ and $V_3=\{ 7, 8\}$, and two singleton sets $\{1\}$, $\{ 3\}$ of non-isolated vertices. The column-gcd's are \[(\bgcd_i)_{i \in [1,8]} = (1,3,1,1,1,1,6,2).\]
For $V $ equal to $\{1\}$ or $\{ 3\}$, we have $r_V=0$ by Corollary \ref{Corol singleton}. 
For $V$ equal to $V_1$ or $V_3$, we have 
\[ r_V = (2^2-1) + (2^1-1) - 2 = 2.\]And for $V_2$, we obtain $r_{V_2} = (2^2 - 1) - 2 = 1$. Then $r =  5$.
\end{exam}

\begin{exam} \label{Exam:StarShaped}
  Consider the cluster algebra associated to the quiver with no frozen vertices
 
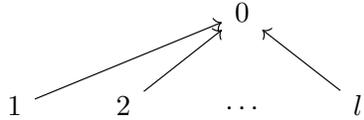
\begin{center} 
  \begin{tikzcd}
      &   & 0 & \\
      1 \ar[urr] & 2 \ar[ur] & \cdots & l \ar[ul].
    \end{tikzcd}
  \end{center}
  The partner sets are $V_0 = \{0\}$ and $V_1 = [1,l]$.
  For the exchange polynomials we have $f_0 = 1 + x_1\cdots x_l$ and $f_1=\dots = f_l = 1 + x_0$.
  Thus
  \[
    r = r_{V_1} = 2^l - l - 1.
  \]
\end{exam}

\begin{rema} Let $l$, $k \in \Z_{\ge 1}$.
  The \emph{Eulerian number} $A(l,k)$ is the number of permutations $\sigma\in S_l$ with exactly $k$ descents.
  Here, a \emph{descent} of $\sigma$ is a pair $(i,i+1)$ such that $\sigma(i)>\sigma(i+1)$.

  Suppose that $k=1$. Then every subset $A\subseteq [1,l]$ defines a permutation
  \[
    \sigma_A=\left(\begin{matrix}1&2&\cdots& r&r+1&r+2&\cdots&l\\a_1&a_2&\cdots&a_r&b_1&b_2&\cdots&b_{l-r}\end{matrix}\right)
  \]
  when we write $A=\{a_1<a_2<\ldots<a_r\}$ and $[1,l]\setminus A =\{b_1<b_2<\ldots<b_{l-r}\}$.
  It is easy to see that $\sigma\in S_l$ has exactly one descent if and only if $\sigma=\sigma_A$ for some set $A$ such that $a_r>b_1$.
  Note that $a_r\leq b_1$ if and only if $A=[1,u]$ for some $u \in [0,l]$. 
  We see that $A(l,1)=2^l-l-1$. Hence the ranks appearing in Example~\ref{Exam:StarShaped} are all Eulerian numbers.
\end{rema}

\begin{exam}
  Building on Example~\ref{Exam:StarShaped}, consider a quiver that is an orientation of a complete bipartite graph:  let $Q_0 = [1,l] \cup [l+1,t]$, and for each $i \in [1,l]$ and $j \in [l+1,t]$, let there be an arrow $i \to j$ in $Q$.
  The partner sets are $V_1 = [1,l]$ and $V_2 = [l+1,t]$.
  For the exchange polynomials we have $f_1 = \dots = f_l = 1 + x_{l+1}\cdots x_t$ and $f_{l+1} =\dots = f_{t} = 1 + x_1 \cdots x_l$.  
Therefore
  \[
    r = r_{V_1}  + r_{V_2} = 2^l - l - 1 + 2^{t-l} -(t-l) - 1.
  \]
\end{exam}

\begin{exam}
  This example illustrates that, in the presence of isolated vertices, the class group can differ depending on whether $K=\Z$ or $K$ is a field.
  Consider the following quiver.
  
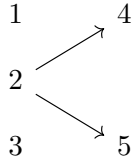
\begin{center}
    \begin{tikzcd}[row sep = small]
      1 & 4 \\
      2 \ar[ru]  \ar[rd] & \\
      3 & 5\\
    \end{tikzcd}
  \end{center}
  Then $V_1 = \{1,3\}$ is the partner set of isolated vertices, $V_2 = \{2\}$ is a singleton, and $V_3=\{4,5\}$.
  We have
  \begin{align*}
    f_1&=f_3=2, &  f_2&=1+x_4x_5, &  f_4&=f_5    = 1 + x_2.
  \end{align*}
  If $K = \Z$, then the class group has rank $r_{V_1} + r_{V_2} + r_{V_3} = 1 + 0 + 1 = 2$.
  Note that all exchange polynomials, including $f_1=f_3=2$ are prime elements in $\Z[x_1,\ldots,x_5]$.

  If $K$ is a field, to apply our results we must first freeze $1$ and $3$ (see Remark~\ref{Rem:FreezeTrivial}).
  The algebra after freezing $1$ and $3$ is isomorphic to the original one since $f_1=f_3=2 \in K^\times$.
  In particular the class group does not change after freezing the vertices.
  We find that the class group has rank $r_{V_2} + r_{V_3} = 0 + 1 = 1$.
\end{exam}

Factoriality of cluster algebras of the simply laced Dynkin types ($A$, $D$, $E$) was first classified in the preprint \cite{L1}.
Our results now allow us to easily recover this classification.
In the following we list the class groups for all cluster algebras of Dynkin or extended Dynkin type.
If $\Sigma$ is a seed of Dynkin type $G$ (without coefficients), and $A=A(\Sigma)$ is the corresponding cluster algebra, we write $\cC(G)=\cC(A)$ for the class group of $A$.

\begin{coro} \label{coro:dynkin}
  For the cluster algebras of Dynkin types over $K$, we obtain the following results:

  \begin{itemize}
  \item The cluster algebra of type $A_n$ is factorial if $n \ne 3$, and $\cC(A_3)\cong \Z$.
  \item In type $B_n$:
  \begin{enumerate}
  \item If $n=2$,
   the cluster algebra of type $B_2$ is factorial if $\mu^*_4(K) = \emptyset$, and $\cC(B_2) \cong \Z$ otherwise.
 \item  If $n=3$, $\cC(B_3) \cong \Z$.
 
 \item If $n>3$, then the cluster algebra of type $B_n$ is factorial.
 \end{enumerate}
   
  \item The cluster algebra of type $C_n$ is factorial if $\mu^*_4(K) = \emptyset$, and $\cC(C_n) \cong \Z$ otherwise.
  \item Type $D_n$ has $\cC(D_n) \cong \Z$ for $n > 4$, and $\cC(D_4) \cong \Z^4$.
  \item The cluster algebras of types $E_6$, $E_7$,~$E_8$ are factorial.
  \item The cluster algebra of type $F_4$ is factorial.
  \item Type $G_2$ has $\cC(G_2) \cong \Z$ if $\mu^*_6(K) = \emptyset$, and $\cC(G_2) \cong \Z^2$ otherwise.
  \end{itemize}
\end{coro}

\begin{proof}
Exchange matrices for Dynkin cluster algebras can be found in \cite[Section 5]{FWZ2}. 

We prove the claim only in type $B_n$. The other statements can be shown in a similar way. The cluster algebra of type $B_2$ is treated in Example \ref{ex_rank_two}.
Assume that $n=3$. Then 
\begin{align*}
B = \begin{bmatrix} 0 & 2 & 0 \\ -1 & 0 & 1 \\ 0 & -1 & 0\end{bmatrix}
\end{align*}
is an exchange matrix of the cluster algebra of type $B_3$. There are two partner sets $V_1=\{1,3\}$ and $V_2=\{2\}$. Using $\bgcd_1=\bgcd_2=\bgcd_3=1$ we see $r=r_{V_1}=2^2-1-2=1$, so that $\cC(B_3)\cong \Z$. If $n>3$, then every partner set is a singleton and $\bgcd_i=1$ for all $i\in[1,n]$ so that $\cC(B_n)\cong 0$ for $n>3$.
\end{proof}

\begin{coro} \label{coro:extendeddynkin} For the cluster algebras of extended Dynkin types over $K$, we obtain the following results:  
 \begin{itemize}
  \item Type $\tilde{A}_n = \tilde{A}_{p,q}$. 
  \begin{enumerate}
  \item The cluster algebra of type $\tilde{A}_{p,q}$ is factorial if $(p,q)\notin\{(1,1),(2,2)\}$.
  \item The cluster algebra of type $\tilde{A}_{1,1}$ is factorial if $\mu^*_4(K) = \emptyset$, and $\cC(\tilde{A}_{1,1})\cong\Z^2$ otherwise.
  \item For type $\tilde{A}_{2,2}$ we have $\cC(\tilde{A}_{2,2})\cong \Z^2$.
  \end{enumerate}
  
  \item Type $\tilde{B}_n$:
  
  \begin{enumerate}
  \item If $n=3$, then $\cC(\tilde{B}_n) \cong \Z^4$.
 \item If $n>3$, then $\cC(\tilde{B}_n) \cong \Z$.
  \end{enumerate} 
  
  \item Type $\tilde{C}_n$:
\begin{enumerate}
\item  If $n=2$, then $\cC(\tilde{C}_2) \cong \Z$ if $\mu^*_4(K) = \emptyset$, and $\cC(\tilde{C}_2) \cong \Z^4$ if $\mu^*_4(K) \neq \emptyset$.
\item If $n>2$, then the cluster algebra of type $\tilde{C}_n$ is factorial if $\mu^*_4(K) = \emptyset$, and $\cC(\tilde{C}_n) \cong \Z^2$ if $\mu^*_4(K) \neq \emptyset$.
\end{enumerate}

  \item Type $\tilde{D}_n$ has $\cC(\tilde{D}_n) \cong \Z^2$ for $n > 4$, and $\cC(\tilde{D}_4) \cong \Z^{11}$.
  \item The cluster algebras of types $\tilde{E}_6$, $\tilde{E}_7$,~$\tilde{E}_8$ are factorial.
  \item The cluster algebra of type $\tilde{F}_4$ is factorial.
  \item Type $\tilde{G}_2$ has $\cC(\tilde{G}_2) \cong \Z$.
  \end{itemize}
\end{coro}

 \begin{proof}
We prove the claim only in type $\tilde{C}_n$. The other statements can be shown in a similar way. The exchange matrices can be obtained from the Cartan matrices as in \cite[Lecture 2, Section 4.2]{FR}, and the diagrams needed can be found in \cite[Section 4.8]{K}. Note that $\nu_K(4)=1$ if $\mu^*_4(K) = \emptyset$, and $\nu_K(4)=2$ otherwise. Assume that $n=2$. Then 
\begin{align*}
B = \begin{bmatrix} 0 & 1 & 0 \\ -2 & 0 & 2 \\ 0 & -1 & 0\end{bmatrix}
\end{align*}
is an exchange matrix of the cluster algebra of type $\tilde{C}_2$. There are two partner sets $V_1=\{1,3\}$ and $V_2=\{2\}$. Using $\bgcd_1=\bgcd_3=2$ and $\bgcd_2=1$ we see $r=r_{V_1}=(2^2-1)\nu_K(4)-2$. Hence $\cC(\tilde{C}_2)\cong \Z$ if $\mu^*_4(K) = \emptyset$, and $\cC(\tilde{C}_2)\cong \Z^4$ if $\mu^*_4(K) \neq \emptyset$. If $n>2$, then every partner set is a singleton. Using $\bgcd_i=1$ for all $i\in[2,n]$ and $\bgcd_1=\bgcd_{n+1}=2$ we see that $r=r_{\{1\}}+r_{\{n+1\}}=2(\nu_K(4)-1)$. Hence $\cC(\tilde{C}_n)\cong 0$ if $\mu^*_4(K) = \emptyset$, and $\cC(\tilde{C}_n)\cong \Z^2$ if $\mu^*_4(K) \neq \emptyset$ in this case.
\end{proof}

The following example shows that every $t \in \Z_{\ge 0}$ can appear as the rank of a class group of an acyclic cluster algebra.
As we have already seen factorial cluster algebras (that is $t=0$), it is sufficient to consider $t \ge 1$.

\begin{exam}
  Let $t \in \Z_{\ge 1}$ and consider the quiver
  
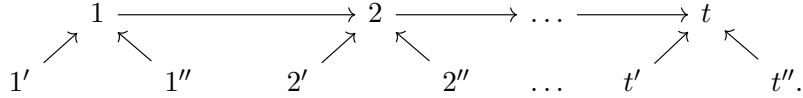
\begin{center}
    \begin{tikzcd}[row sep = small, column sep=small]
         & 1 \ar[rrrr] &    & &    & 2 \ar[rr] &   & \dots  \ar[rr]  &    & t & \\
      1' \ar[ru] &  & 1'' \ar[lu] & & 2' \ar[ru] &   & 2'' \ar[lu] & \dots & t' \ar[ru] & & t'' \ar[lu].
    \end{tikzcd}
  \end{center}
  That is, to each vertex $i$ of the linear quiver on $t$ vertices we attach two new vertices $i'$, $i''$.
  The partner sets are $\{i',i''\}$ for $i \in [1,t]$, and the singleton sets $\{i\}$ for $i \in [1,t]$.
  By Corollary~\ref{teorema:rank}, then $r_{\{i',i''\}} = 1$ and $r_{\{i\}} = 0$.
  We conclude that the class group of the associated cluster algebra has rank $t$.
\end{exam}

\section{A cluster algebra that is not a Krull domain: Markov quiver}
\label{sec:markov}

For locally acyclic cluster algebras, it is known that they are noetherian and integrally closed, and therefore Krull domains.
It seems that so far no example of a cluster algebra that is not a Krull domain has been exhibited.
In this section we show that the cluster algebra associated to the Markov quiver, which also provides a wealth of other counterexamples, is not a Krull domain.

The Markov quiver is 
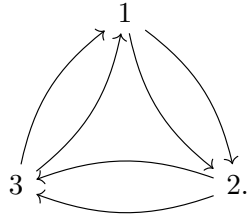
\[
  \begin{tikzcd}[row sep=50pt]
    & 1 \ar[dr,bend left=20] \ar[dr, bend right=20] & \\
    3 \ar[ur, bend left=20] \ar[ur, bend right=20] & & 2. \ar[ll, bend left=20] \ar[ll, bend right=20]
  \end{tikzcd}
\]
Any mutation of the Markov quiver is again isomorphic to the Markov quiver.
It is known that the cluster algebra $A$ associated to the Markov quiver (with base ring $\Z$) is non-noetherian, not equal to its upper cluster algebra, and not finitely generated (see \cite{M}).
The cluster algebra is $\Z_{\ge 0}$-graded by restricting the grading by total degree from a Laurent-polynomial ring.
(All cluster variables of this infinite-type cluster algebra have degree $1$ in this grading).

The upper cluster algebra $\cU$ of $A$, by contrast, is very well-behaved (see \cite{MM}).
It is given by $\Z[x_1,x_2,x_3,M]$ with
\[
  M = \frac{x_1^2 + x_2^2 + x_3^2}{x_1x_2x_3}.
\]
The upper cluster algebra $\cU$ can also be presented by generators $x_1$, $x_2$, $x_3$,~$M$ with the single relation $x_1x_2x_3M - x_1^2 - x_2^2 - x_3^2 = 0$.
The representation of the element $M$ is in fact independent of the chosen cluster, that is, if $\{y_1,y_2,y_3\}$ is another cluster, then also
\[
  M = \frac{y_1^2 + y_2^2 + y_3^2}{y_1y_2y_3}.
\]

In \cite{MM} also some height-$2$ prime ideals of $\cU$ are computed (the indices should be interpreted modulo $3$):
\[
  \langle x_1,x_2,x_3 \rangle_{\cU},\  \langle x_i, x_{i-1}^2 + x_{i+1}^2, M \rangle_{\cU} \text{ for $i \in [1,3]$}.
\]

Note that $\langle x_i, x_{i-1}^2 + x_{i+1}^2, M \rangle_{\cU} = \langle x_i, M \rangle_{\cU}$ due to $x_{i-1}^2 + x_{i+1}^2 = x_i x_i'$.
Another important relation to keep in mind is
\[
  x_i + x_i' = x_{i-1}x_{i+1}M.
\]

The elements $x_1$, $x_2$, $x_3$ are prime elements of $\cU$ (we use here that the base ring is $\Z$).
Hence, the upper cluster algebra $\cU$ is factorial by Nagata's Theorem.
By symmetry, every cluster variable of $A$ is a prime element of $\cU$.

\begin{lemm} \label{l-krull-non-upper}
  If $A$ is a cluster algebra with $A\ne \cU$ and $A$ is a Krull domain, then there exists a $\prim \in \Xx(A)$ such that $\prim$ contains at least one cluster variable of each cluster.
\end{lemm}

\begin{proof}
  For a cluster $\mathbf y = \{y_1,y_2,y_3\}$ let $A_\mathbf y = A[y_1^{-1},y_2^{-1},y_3^{-1}]$.
  Let $C$ denote the set of all clusters.
  If $\prim \in \Xx(A)$ with $\mathbf y \cap \prim = \emptyset$, then $\prim A_{\mathbf y}$ is a prime ideal of $A_{\mathbf y}$ and hence $A_\prim = (A_{\mathbf y})_{\prim A_{\mathbf y}}$.

  We proceed with a proof by contradiction.
  Suppose that, for every $\prim \in \Xx(A)$, there is a cluster $\mathbf y \in C$ with $\mathbf y \cap \prim = \emptyset$.
  Since $A$ is a Krull domain
  \[
    A = \bigcap_{\prim \in \Xx(A)} A_\prim = \bigcap_{\mathbf y \in C} \bigcap_{\substack{\prim \in \Xx(A)\\\prim \cap \mathbf y = \emptyset}} A_\prim = \bigcap_{\mathbf y \in C} \bigcap_{\substack{\prim \in \Xx(A)\\\prim \cap \mathbf y = \emptyset}} (A_{\mathbf y})_{\prim A_{\mathbf y}} = \bigcap_{\mathbf y \in C} A_{\mathbf y} = \cU,
  \]
  a contradiction.
\end{proof}

\begin{center}
  \emph{For the remainder of this section, let $A$ be the Markov cluster algebra.}
\end{center}

Let $A_1 = \langle\{\, x \mid \text{$x$ is a cluster variable} \,\}\rangle_A = \{\, f \in A \mid \text{$f$ has constant term $0$} \,\}$.

Let $(x_1,x_2,x_3)$ be the initial cluster.
Note that the labeling of the initial cluster induces on every other cluster a natural labeling $(y_1,y_2,y_3)$ that is consistent with that of the initial cluster,  so that $y_j$ corresponds to $x_j$.
For $i \in [1,3]$, let
\[
  \prim(i) = \langle \{\, y_i, y_{i-1} + y_{i-1}', y_{i+1} + y_{i+1}' \ \vert\ \text{$(y_1,y_2,y_3)$ a cluster} \,\} \rangle_A,
\]
where the clusters are labeled consistently with the initial cluster.

\begin{lemm} \label{l-markov-primes}
  Let $(x_1,x_2,x_3)$ be a cluster of $A$.
  \begin{enumerate}
  \item\label{l-markov-primes:a1} $A_1 = \langle x_1,x_2,x_3 \rangle_{\cU} \cap A = A_1\cU \cap A$.
  \item\label{l-markov-primes:pi}  $\prim(i) = \langle x_i, M \rangle_{\cU} \cap A = \prim(i) \cU \cap A$ for $i \in [1,3]$.
  \end{enumerate}
  In particular, $A_1$, $\prim(1)$, $\prim(2)$, and $\prim(3)$ are prime ideals of $A$.
\end{lemm}

\begin{proof}
  \ref*{l-markov-primes:a1}
  First note $A_1\cU =\langle x_1,x_2,x_3 \rangle_{\cU}$: The inclusion ``$\supseteq$'' is trivial.
  Due to $x_i+x_i' = x_{i-1}x_{i+1}M$ we have $x_i' \in \langle x_1,x_2,x_3 \rangle_{\cU}$ for all $i \in [1,3]$.
  Inductively, $A_1\cU \subseteq \langle x_1,x_2,x_3 \rangle_{\cU}$.
  Therefore also $A_1 \cU \cap A = \langle x_1, x_2, x_3 \rangle_\cU \cap A$.

  Thus we have a ring homomorphism $\varphi\colon A/A_1 \to \cU/\langle x_1,x_2,x_3\rangle_{\cU} \cong \Z[M]$, where the codomain is a polynomial ring in the indeterminate $M$.
  Note that $\im\varphi \cong \Z$.
  Since also $A/A_1 \cong \Z$, we see that $\varphi$ is injective, and hence $A_1 = \langle x_1,x_2,x_3\rangle_{\cU} \cap A$.

  \ref*{l-markov-primes:pi}
  Without loss of generality we consider the case $i=1$.
  Clearly $\cU / \langle x_1, M \rangle_\cU \cong \Z[x_{2},x_{3}]/\langle x_{2}^2 + x_{3}^2\rangle$.
  We first show $\prim(1) \subseteq \langle x_1, M \rangle_\cU$, and do so inductively.
  Suppose $(y_1,y_2,y_3)$ is a cluster.
  Since $y_j + y_j' = y_{j-1}y_{j+1}M$, we have $y_j + y_j' \in \langle x_1, M \rangle_\cU$.
  If $y_1 \in \langle x_1, M \rangle_\cU$ then this implies that also $y_1' \in \langle x_1, M \rangle_\cU $.
  Hence $\prim(1)\cU \subseteq \langle x_1, M \rangle_\cU$.

  Due to $x_2$, $x_3 \in A$, the ring homomorphism $A/\prim(1) \to \cU / \langle x_1, M \rangle_\cU$ is surjective.
  If we can show $A/\prim(1) \cong \Z[x_2,x_3]/\langle x_2^2 + x_3^2\rangle$, we can conclude $\prim(1) = \langle x_1,M \rangle_\cU \cap A$ and are done.
  
  We have $x_2^2 + x_3^2 = x_1x_1' \in \prim(1)$.
  Moreover, for any cluster $(y_1,y_2,y_3)$ we have $y_j + y_j' \in \prim(1)$, so that $y_j' \equiv -y_j \mod \prim(1)$.
  This shows that $A/\prim(1)$ is generated by $x_2$ and $x_3$ as a $\Z$-algebra.
  Thus $A/\prim(1)$ is a factor ring of $\Z[x_2,x_3]/\langle x_2^2+x_3^2\rangle$.
  Since $A/\prim(1)$ maps surjectively onto $\Z[x_2,x_3]/\langle x_2^2+x_3^2\rangle$, we must in fact have $A/\prim(1) \cong \Z[x_2,x_3]/\langle x_2^2+x_3^2\rangle$.
\end{proof}

\begin{lemm} \label{l-markov-twovar}
  If a prime ideal $\prim \subseteq A$ contains two cluster variables of the same cluster, it contains all cluster variables of $A$.
\end{lemm}

\begin{proof}
  We first show:
  If a prime ideal $\prim \subseteq A$ contains two cluster variables of the same cluster, it contains the third.
  Without restriction, let  $x_1$,~$x_2 \in \prim$.
  Let $(x_1',x_2,x_3)$ be the cluster after mutation in direction $1$, so that $x_1'x_1 = x_2^2 + x_3^2$.
  From this relation we see $x_3^2 \in \prim$, hence $\prim$ contains the entire cluster $\{x_1,x_2,x_3\}$.

  The claim now follows easily by an inductive argument:
  If $\prim$ contains a cluster $\{x_1,x_2,x_3\}$, after mutation in any direction, it still contains two variables of the mutated cluster, and hence also the new variable.
\end{proof}

\begin{lemm}
  Let $\prim \subseteq A$ be a prime ideal such that $\prim$ contains one variable of each cluster.
  Then $\prim$ contains one of $A_1$, $\prim(1)$, $\prim(2)$, or $\prim(3)$.
\end{lemm}

\begin{proof}
  If $\prim$ contains two variables of one cluster, then $\prim = A_1$ by Lemma~\ref{l-markov-twovar}.
  We may assume that $\prim$ contains exactly one variable of each cluster.
  Fix a cluster $\{x_1,x_2,x_3\}$ and suppose $x_1 \in \prim$.
  Then also $x_1' \in \prim$ since $\prim$ contains exactly one variable of $\{x_1',x_2,x_3\}$, and $x_2$,~$x_3 \not\in \prim$.
  Inductively, we find that for each cluster, $\prim$ contains the variable corresponding to $x_1$ of our initial cluster.

  Thus, for a cluster $(y_1,y_2,y_3)$, with labeling consistent with our initial seed, we have $y_1 \in \prim$.
  Then also $y_2^2 + y_3^2 = y_1 y_1' \in \prim$.
  Hence $y_2(y_2+y_2') = y_1^2 + y_2^2 + y_3^2 \in \prim$.
  Since $y_2 \not \in \prim$, we must have $y_2 + y_2' \in \prim$.
  Similarly, we conclude $y_3 + y_3' \in \prim$.
  Altogether $\prim(1) \subseteq \prim$.
\end{proof}

\begin{theo} \label{Thm:Markov}
  The cluster algebra $A$ associated to the Markov quiver is not a Krull domain.
\end{theo}

\begin{proof}
  Suppose to the contrary that $A$ is a Krull domain.
  By Lemma~\ref{l-krull-non-upper}, there exists a height-$1$ prime ideal $\prim \subseteq A$ containing one variable of each cluster.
  Hence at least one of $A_1$, or $\prim(i)$ for $i \in [1,3]$ must be a height-$1$ prime ideal of $A$.
  We show that this is false, by showing that each of these ideals properly contains the prime ideal $x_1 \cU \cap A$.
  Recall that, in $\cU$, each cluster variable is prime.
  In particular, $x_1\cU$ is a prime ideal and does not contain any other cluster variable.

  \emph{Case $A_1$:} Since $A_1 = A_1\cU \cap A$ and $x_1\cU \subseteq A_1 \cU$, we have $x_1 \cU \cap A \subseteq A_1$.
  However, $x_2 \in A_1 \setminus x_1\cU$.

  \emph{Case $\prim(1)$:} Since $\prim(1) = \prim(1)\cU \cap A$ and $x_1 \cU \subseteq \prim(1) \cU$, we have $x_1\cU \cap A \subseteq \prim(1)$.
  However, $x_1' \in \prim(1) \setminus x_1\cU$.
\end{proof}

\begin{lemm}
  None of the cluster variables of $A$ is a prime element.
\end{lemm}

\begin{proof}
  If one cluster variable is prime, then all are, by symmetry.
  Due to Nagata's Theorem, the algebra $A$ is then factorial, and hence a Krull domain.
  This contradicts what we just showed.

  Alternatively, and more directly, if $(x_1,x_2,x_3)$ is a cluster we see
  \[
    x_i(x_i+x_i') = x_1^2 + x_2^2 + x_3^2,
  \]
  for all $i \in [1,3]$.
  Since we know that cluster variables are non-associated atoms, we see that $A$ is not prime.
\end{proof}

\begin{rema}
  Since $A$ is not a Krull domain, it cannot be both $v$-noetherian and completely integrally closed.
  We do not know which of these properties fail(s).
\end{rema}

\section{Non-invertible frozen variables}
\label{sec:noninvertible}

Throughout the paper we have so far assumed that in a cluster algebra all frozen variables are invertible.
In general one is also interested in cases where only some (or none) of the frozen variables are invertible.
In this section we briefly consider this situation.
It turns out that Theorem~\ref{Thm:MainGeneral} extends straightforwardly, while the issue of Theorem~\ref{Thm:MainAcyclic} (and hence Theorem~\ref{Thm:MainAcyclicSimplified}) is more subtle.

First we need to revisit the definition of a cluster algebra (Definition~\ref{def:clusteralgebra}).
For a tuple of frozen variables $\y$ we tacitly use the notation $\inv \subseteq \y$ to denote a subset of frozen variables.
Given a seed $\Sigma=(\x,\y,B)$ and such a subset $\inv \subseteq \y$ of frozen variables, let again $\mathcal{X}$ denote the set of all exchangeable variables in a seed equivalent to $\Sigma$, and define
\[
  A(\Sigma,\inv) = K[x,y,z,z^{-1} \mid x \in \mathcal{X},\, y \in \y,\, z \in \inv] \subseteq \mathcal{F}(\Sigma).
\]
The difference to Definition~\ref{def:clusteralgebra} is that now only the frozen variables in $\inv$ are invertible, while those in $\y \setminus \inv$ are \emph{non-invertible frozen variables}.
The case that we have dealt with so far corresponds to all frozen variables being invertible, that is $\inv = \y$.

For $A = A(\Sigma,\inv)$ with $\inv \subseteq \y$ let us define the mixed Laurent/polynomial ring
\[
  A_\x = A[x^{-1} \mid x \in \x] = K[x,x^{-1},y,z,z^{-1} \mid x \in \x,\, y \in \y,\, z \in \inv] \subseteq \mathcal{F}(\Sigma).
\]
A refined version of the Laurent phenomenon implies $A \subseteq A_\x$ (see \cite[Theorem 3.3.6]{FWZ}).
Geiss--Leclerc--Schröer, in \cite{GLS}, consider this more general case as well, and so Theorem~\ref{Thm:ClusterAtoms} holds with the following modification: the group of units of $A$ is $A^\times = K^\times \times \langle x^{\pm 1}, z^{\pm 1} \mid x \in \x,\, z \in \inv \rangle$.
The non-invertible frozen variables are atoms in $A$ because they are atoms in $A_\x$ and $A^\times=A_\x^\times$.
Corollary~\ref{Cor:ClusterFactorial} carries over straightforwardly to this setting.

We obtain a generalization of Theorem~\ref{Thm:MainGeneral}.
\begin{theo} \label{Thm:MainGeneralNonInvertible}
  Let $\Sigma = (\x,\y,B)$ be a seed with exchangeable variables $\x=(x_1,\ldots, x_n)$ and frozen variables $\y=(x_{n+1}, \ldots, x_{n+m})$.
  Let $A=A(\Sigma,\inv)$ be the cluster algebra associated to $\Sigma$ in which \textup{(}only\textup{)} the frozen variables $\inv \subseteq \y$ are invertible.
  Suppose that $A$ is a Krull domain, and let $t \in \Z_{\ge 0}$ denote the number of height-$1$ prime ideals that contain one of the exchangeable variables $x_1$, $\ldots\,$,~$x_n$.
  Then the class group of $A$ is a free abelian group of rank $t - n$.

  If $n+m > 0$, that is $A \ne K$, then each class contains exactly $\card{K}$ height-$1$ prime ideals.
\end{theo}

\begin{proof}
  We note that Lemma~\ref{Lemma:IndicatorPrimes} holds true in this setting with the same proof.
  
  Now $A[x_1^{-1},\ldots,x_n^{-1}] = A_\x = D[x_1^{\pm 1}, \ldots, x_n^{\pm 1}]$ is a Laurent polynomial ring in the indeterminates $x_1$, $\ldots\,$,~$x_n$ over the mixed Laurent/polynomial ring $D=K[y,z,z^{-1} \mid y \in \y, z \in \inv]$.
  Since $D$ is factorial, we may apply Theorem~\ref{Thm:CG} to the cluster algebra $A$ and the exchangeable variables $\{x_1,\ldots,x_n\}$.
  Thus, the class group $\mathcal C(A)$ is again generated by $t$ elements with $n$ relations, stemming from the factorizations of $x_1A$, $\ldots\,$,~$x_nA$ as divisorial products of height-$1$ prime ideals.

  The proof now proceeds as the one of Theorem~\ref{Thm:MainGeneral} (on page \pageref{Proof:MainGeneral}).
\end{proof}

The situation of acylic seeds (Theorem~\ref{Thm:MainAcyclic}) is more complicated.
The work in \cite{BFZ,M} that we build upon assumes that frozen variables are invertible.
This factors into our work in several ways.

\begin{enumerate}
  \item
    To show that (locally) acyclic cluster algebras with $\inv=\y$ are Krull domains, one uses that they are equal to their upper cluster algebras, see \cite{M2}.
    Using the refined Laurent phenomenon, we may define the upper cluster algebra of $A=A(\Sigma,\inv)$ as
    \[
      \mathcal U = \mathcal U(\Sigma,\inv) = \bigcap_{\Sigma'=(\x',\y,B')} A_{\x'},
    \]
    where the intersection is taken over all seeds $\Sigma'$ mutation-equivalent to $\Sigma$.

    Whether a given cluster algebra is equal to its upper cluster algebra is an important theme in the study of cluster algebras.
    We refer to the introduction of \cite{GY2} for a good overview on current results.
    In contrast to the case $\inv=\y$, if $\inv \ne \y$, then there exist locally acyclic cluster algebras that are \emph{not} equal to their upper cluster algebra.
    An example can be found in \cite[Proposition 4.1]{BMS}, based on \cite{GSV}.

    However, even when $\inv \ne \y$, there are still many families of cluster algebras known that coincide with their upper cluster algebras.
    In \cite{GY2}, Goodearl--Yakimov construct cluster algebra structures on symmetric Poisson nilpotent algebras (under mild conditions), and show that these cluster algebras are equal to their upper cluster algebras.

    On the combinatorial side, Bucher--Machacek--Shapiro, in \cite{BMS}, give sufficient criteria for $A=\mathcal U$ in the case where $\inv \ne \y$.
    They show $A=\mathcal U$ for \emph{locally isolated} cluster algebras.
    If $\y=\inv$, then a cluster algebra is locally isolated if and only if it is locally acyclic, but in general the former condition is more restrictive than the latter.
  
    A seed $\Sigma=(\x,\y,B)$ is \emph{source-freezing} (with respect to $\inv \subseteq \y$) if all arrows between an exchangeable index and non-invertible frozen index point \emph{from} the exchangeable index \emph{to} the frozen index.
    In other words, if $i$ is the index of a non-invertible frozen variable, then the entries of the $i$-th row of the exchange matrix $B$ are non-positive.
    If $\Sigma$ is acyclic source-freezing, then $A(\Sigma,\inv) = \mathcal{U}(\Sigma,\inv)$ by \cite[Theorem 3.7 and Corollary 3.8]{BMS}.
  \item Theorem~\ref{Thm:BFZ}, giving a presentation for acyclic cluster algebras, is used in Section~\ref{sec:acyclic-prime} and assumes $\y=\inv$.
  \item The proof of the important Lemma~\ref{Lemma:Muller} crucially assumes $\y=\inv$.
\end{enumerate}

That said, the following lemma in combination with Nagata's Theorem sometimes still allows a reduction to the case of invertible frozen variables.

\begin{lemm} \label{lem:killnoninv}
  Let $\Sigma=(\x,\y,B)$ be a seed with $\x=(x_1,\ldots,x_n)$ and $\y=(x_{n+1},\ldots, x_{n+m})$.
  Let $\inv \subseteq \y$ and $A =A(\Sigma,\inv)$.
  Then every non-invertible frozen variable is a prime element in the upper cluster algebra of $A$.
\end{lemm}

\begin{proof}
  By relabeling, we may suppose $x_{n+1}$, $\ldots\,$,~$x_{n+p} \not\in A^\times$ and $x_{n+p+1}$, $\ldots\,$,~$x_{n+m} \in A^\times$ for some $p \in [0,m]$.
  Let $i \in [n+1,n+p]$ so that $x_i$ is a non-invertible frozen variable.
  Let $\cU=\cU(\Sigma,\inv)$ be the upper cluster algebra of $A$.
  Let $a$,~$b \in \cU$ and suppose that $x_i$ divides $ab$.
  We have to show that $x_i$ divides $a$ or $b$ in $\cU$.

  Since $x_i$ is prime in $A_\x$, the element $x_i$ must divide $a$ or $b$ in $A_\x$.
  Without restriction, assume that $x_i$ divides $a$ in $A_\x$, in other words $a/x_i \in A_\x$.
  We shall show that then $a/x_i \in A_{\x'}$ for every seed $(\x',\y,B')$.
  Then $a/x_i \in \cU$, that is, the element $x_i$ divides $a$ in $\cU$.

  Since any two seeds are connected by a sequence of mutations, it suffices to show:
  If $(\x,\y,B)$ is any seed with $a/x_i \in A_\x$, and $(\x',\y,B)$ is obtained from $(\x,\y,B)$ by mutation at $k \in [1,n]$, then also $a/x_i \in A_{\x'}$.
  To simplify the notation, we may assume $k=1$.
  Since $a \in x_i A_{\x}$ and  $a \in \cU \subseteq A_{\x'}$, we may write
  \begin{equation} \label{e:laurent}
    a = \frac{x_i f(x_1,x_2,\ldots,x_{n+m})}{x_1^rM^r} = \frac{g(x_1',x_2,\ldots,x_{n+m})}{(x_1')^rM^r},
  \end{equation}
  with $f \in K[x_1,x_2,\ldots,x_{n+m}]$, with $g \in K[x_1',x_2,\ldots,x_{n+m}]$, with $M=x_2\cdots x_nx_{n+p+1}\cdots x_{n+m}$ and with $r \in \Z_{\ge 0}$.
  Since $f \in A \subseteq A_{\x'}$ we may further write $f = f'(x_1',x_2,\ldots,x_{n+m}) / (x_1')^s M^s$ with $f' \in K[x_1',x_2,\ldots,x_{n+m}]$ and $s \in \Z_{\ge 0}$.
  Moreover $x_1 x_1' = f_1(x_2,\ldots,x_n)$ with $f_1 \in K[x_2,\ldots,x_n]$ the exchange polynomial of $x_1$.
  Substituting into Equation~\eqref{e:laurent} to eliminate $x_1$, and clearing denominators,
  \[
    (x_1')^{2r} x_i f' = (x_1')^s M^s f_1^r g.
  \]
  All the factors in this expression are contained in the polynomial ring $K[x_1',x_2,\ldots,x_{n+m}]$, in which $x_i$ is a prime element.
  Since $x_i$ does not divide the exchange polynomial $f_1$ on the right side, it must divide $g$.
  Then $g = x_i \tilde{g}$ with $\tilde g \in K[x_1',x_2,\ldots,x_{n+m}]$ and thus $a/x_i = \tilde{g} / (x_1')^rM^r \in A_{\x'}$.
\end{proof}

\begin{prop}
  Let $\Sigma = (\x,\y,B)$ be a seed with exchangeable variables $\x=(x_1,\ldots, x_n)$ and frozen variables $\y=(x_{n+1}, \ldots, x_{n+m})$.
  Let $A(\Sigma,\inv)$ be the cluster algebra associated to $\Sigma$ in which \textup{(}only\textup{)} the frozen variables $\inv \subseteq \y$ are invertible.
  Suppose that $A(\Sigma,\inv)$ is a Krull domain and that it is equal to its upper cluster algebra $\mathcal U(\Sigma,\inv)$.
  Then
  \[
    \mathcal{C}(A(\Sigma,\inv)) \cong \mathcal{C}(A(\Sigma,\y)).
  \]
\end{prop}

\begin{proof}
  Note that $A(\Sigma,\y) = A(\Sigma,\inv)[y^{-1} \mid y \in \y\setminus \inv]$.
  By Lemma~\ref{lem:killnoninv} all non-invertible frozen variables are prime elements of $A(\Sigma,\inv)$.
  Thus Nagata's Theorem implies the claim.
\end{proof}

As $A(\Sigma,\y)$ is a cluster algebra in which all frozen variables are invertible, our earlier results can be applied to it.
Thus the computation of the class group for $A(\Sigma,\inv)$ may sometimes still be carried out.
For instance, this is the case if $\Sigma$ is an acyclic source-freezing seed (with respect to $\inv$).
Then $A(\Sigma,\inv) = \mathcal U(\Sigma,\inv)$, and as in \cite[Theorem 4.2]{M}, one sees that $A(\Sigma,\inv)$ is a finitely generated $K$-algebra, and in particular noetherian.
Hence $A(\Sigma,\inv)$ is a Krull domain with $\mathcal{C}(A(\Sigma,\inv)) \cong \mathcal{C}(A(\Sigma,\y))$, and we can compute the rank of the class group directly from the exchange matrix, as in Theorem~\ref{Thm:MainAcyclic}.

\begin{rema}
Outside these situations, we do not know how the choice of $\inv$ changes the factorization theoretical properties of a cluster algebra $A$.
In particular, we do not know whether non-invertible frozen variables are always prime elements in $A$.
In light of Lemma~\ref{lem:killnoninv} and the fact that they are always atoms (a necessary condition for them to be prime elements), it is tempting to hope so.
\end{rema}

\section{Further directions}
\label{sec:further}

While we have a satisfying answer to the computation of a class group of an acyclic seed, and some knowledge on the structure of the class group of a cluster algebra that is a Krull domain in general, several open questions remain.
We record some of them in this final section.

\begin{itemize}
\item We know how to determine the class group of a cluster algebra with acyclic seed (if all frozen variables are invertible); for a cluster algebra that is a Krull domain and has a finite presentation, we at least have an algorithm based on the primary decomposition of the variables of the initial seed.

  Several interesting examples of cluster algebras are not acyclic, but are locally acyclic.
  It would be interesting to have a procedure for computing the class group in these cases, in terms of quiver combinatorics.

\item While we know that every locally acyclic cluster algebra (with invertible frozen variables) is a Krull domain, and not every cluster algebra is a Krull domain, we lack an exact classification of which cluster algebras are Krull domains.

\item Any divisor closed submonoid of a Krull domain is itself a Krull monoid.
  It may be interesting to investigate the divisor-closed submonoid of a cluster algebra generated by its initial cluster, respectively, by all cluster variables.

\item Any Krull domain $A$ possesses a \emph{transfer homomorphism} to a \emph{monoid of zero-sum sequences} $\mathcal B(G_0)$, where $G_0$ is the subset of the class group of $A$ containing height-$1$ prime ideals (see, for instance, \cite[Chapter 3.2]{GHK} or \cite{G}).
  The atoms in $\mathcal B(G_0)$ are the minimal zero-sum sequences over $G_0$.
  If $A$ is an cluster algebra, then each cluster variable is an atom, and hence gives rise to such a minimal zero-sum sequence.
  It may be interesting to see which minimal zero-sum sequences arise in this way.
\end{itemize}

\bigskip
\noindent
\textbf{Acknowledgments.}
\phantomsection%
\addcontentsline{toc}{section}{Acknowledgments}%
We thank Alfred Geroldinger for pointing out the paper \cite{K1} to us and Kaveh Mousavand for raising the question about the irreducibility of $F$-polynomials.
We also thank the reviewers; their suggestions have undoubtedly improved the paper.
In particular, it was suggested by one of the reviewers to discuss non-invertible frozen variables (Section~\ref{sec:noninvertible}).

The first author was supported by FWF grant P30549-N26.
The second author was supported by EPSRC grants EP/N005457/1 and EP/M004333/1.
The third author was supported by the Austrian Science Fund (FWF) projects P26036-N26 and J4079-N32.
The first and the third author were supported by NAWI Graz.

\bibliographystyle{hyperalphaabbr}
\bibliography{krull_cluster}

\end{document}